\documentclass{article}
\usepackage{amsmath,amsthm,amssymb,amsfonts, graphicx}
\usepackage{graphics}
\usepackage{authblk}
\usepackage{upgreek}
\usepackage{amsrefs,hyperref}
\usepackage{cleveref}
\theoremstyle{plain}
\newtheorem{theorem}{Theorem}[section]
\newtheorem{corollary}[theorem]{Corollary}

\newtheorem{lemma}[theorem]{Lemma}
\newtheorem{proposition}[theorem]{Proposition}

\newtheorem{question}[theorem]{Open Question}
\theoremstyle{definition}
\newtheorem{definition}[theorem]{Definition}
\theoremstyle{remark}
\newtheorem{remark}[theorem]{Remark}
\numberwithin{equation}{section}
\newcommand{\diff}{\mathop{}\!\mathrm{d}}

\DeclareMathOperator{\spn}{span}

\title{Permutation symmetric solutions of the incompressible Euler equation}
\author[1]{Evan Miller}
\affil[1]{University of Alabama in Huntsville,
Department of Mathematical Sciences

epm0006@uah.edu}

\begin{document}

\maketitle

\begin{abstract}
In this paper, we study permutation symmetric solutions of the incompressible Euler equation. We show that the dynamics of these solutions can be reduced to an evolution equation on a single vorticity component $\omega_1$, and we characterize the relevant constraint space for this vorticity component under permutation symmetry.
We also give single vorticity component versions of the energy equality, Beale-Kato-Majda criterion, and local wellposedness theory that are specific to the permutation symmetric case.
This paper is significantly motivated by a recent work of the author \cite{MillerRestricted}, which proved finite-time blowup for smooth solutions of a Fourier-restricted Euler model equation, where the Helmholtz projection is replaced by a projection onto a more restrictive constraint space. The blowup solutions for this model equation are odd, permutation symmetric, and mirror symmetric about the plane $x_1+x_2+x_3=0.$ Using the blowup solution introduced by Elgindi in \cite{Elgindi}, we are able to prove there are $C^{1,\alpha}$ solutions of the full Euler equation that blowup in finite-time, which are odd, permutation symmetric, and mirror symmetric about the plane $x_1+x_2+x_3=0$.

We will also prove that divergence-free vector fields that are odd, permutation symmetric, and mirror symmetric about the plane $x_1+x_2+x_3=0$ ($\mathcal{G}_\sigma$ symmetric) are equivalent up to a change of coordinates given by a rotation 
to divergence-free vector fields that are mirror symmetric about each of the three coordinate axes and symmetric with respect to rotations by $\frac{\pi}{3}$ in the horizontal plane ($\mathcal{G}$-symmetric).
The latter discrete symmetry group allows for a Fourier series expansion in cylindrical coordinates that shines a further light on the structure of these symmetry groups, in particular their relation to axisymmetric, swirl-free vector fields.
This yields a useful Anstaz for the further study of the dynamics of the Euler equation under this set of discrete symmetries.
\end{abstract}

\tableofcontents

\section{Introduction}

The incompressible Euler equation is an evolution equation on the space of divergence free vector fields,
\begin{equation}
\partial_t u+\mathbb{P}_{df}((u\cdot\nabla)u)=0,
\end{equation}
where $\mathbb{P}_{df}$ is the Helmholtz projection enforcing the constraint $\nabla\cdot u=0$.
The dynamics of the Euler equation can also be studied in terms of the vorticity $\omega=\nabla\times u$, which satisfies
the evolution equation
\begin{equation}
\partial_t \omega
+(u\cdot\nabla)\omega
-(\omega\cdot\nabla)u=0.
\end{equation}
These two formulations are equivalent, because the velocity can be recovered from the vorticity via the Biot-Savart law
\begin{equation}
    u=\nabla\times(-\Delta)^{-1}\omega.
\end{equation}

In the recent work \cite{MillerRestricted}, the author proved finite-time blowup for smooth solutions of a model equation for the Euler and hypodissipative Navier--Stokes equation on the three dimensional torus.
In this model equation, the Helmholtz projection is replaced by a new projection 
$\mathbb{P}_\mathcal{M}$ onto a subspace of divergence free vector fields
\begin{equation}
        H^s_{\mathcal{M}} \subset
    H^s_{df},
\end{equation}
yielding the Fourier-restricted Euler equation
\begin{equation}
    \partial_t u
    +\mathbb{P}_\mathcal{M}(u\cdot\nabla)u
    =0.
\end{equation}
The $(u\cdot\nabla)u$ nonlinearity is unaltered,
and consequently, while the change in projection dramatically simplifies the dynamics, the mechanism for the cascade of energy to higher order Fourier modes observed in the Fourier-restricted model equation also exists for the full Euler equation.
Because the the dynamics are much more complicated for the full Euler equation, the methods used to prove blowup for the model equation cannot be generalized to the full Euler equation. Nonetheless, the geometric setting in which blowup occurs for the Fourier-restricted Euler model equation remains relevant as a potential scenario for the blowup of the full Euler equation.

The blowup solutions for the restricted Euler equation are odd, permutation symmetric, and mirror symmetric about the plane $x_1+x_2+x_3=0$ (which we will refer to as $\sigma$-mirror symmetric).
The permutation symmetry is essential to the structure of the construction, which suggests that studying the finite-time blowup problem for permutation symmetric solutions of the full Euler equation is worthwhile line of inquiry for the important open problem of finite-time blowup for the Euler equation.
In this paper, we will develop the theory of permutation symmetric solutions of the Euler equation, including developing a Biot-Savart law in terms of one vorticity component and a reduction of the vorticity equation to a scalar evolution equation for one vorticity component. A permutation symmetric solution is a solution of the Euler equation that is invariant when interchanging the coordinate axes. A more formal definition will be given in \Cref{PermSymDef}, and a formal definition of $\sigma$-mirror symmetry for vector fields will be given in \Cref{MirrorSymDef}.

The local existence theory for strong solutions of the Euler equation is very well established, including local existence, uniqueness, and continuous dependence on initial data, but the question of finite-time blowup remains largely open in three dimensions. 
Somewhat counter-intuitively, both the strongest blowup result and the strongest global regularity result known for three dimensional solutions of the Euler equation on the whole space come in the axisymmetric, swirl-free geometry.
The foundational work of Ukhovskii and Yudovich \cite{Yudovich} in the 1960s showed that sufficiently smooth solutions of the Euler equation must exist globally-in-time in this geometry. Very recently, Elgindi \cite{Elgindi} and Elgindi, Ghoul, and Masmoudi \cite{ElgindiGhoulMasmoudi}, proved that there are axisymmetric, swirl-free, $C^{1,\alpha}$ solutions of the incompressible Euler equation that blowup in finite-time. Elgindi and Pasqualotto further generalized this result to include a blowup that is not at the origin, and therefore not at a stagnation point \cite{ElgindiPasqualotto}.
These recent results are the only finite-time blowup results for strong solutions of the Euler equation on the whole space; although in the case of domains with a boundary, there are other blowup results involving finite-time singularity formation at the boundary \cites{ElgindiJeongCorners,ChenHouA,ChenHouB}.
It should be pointed out that while the blowup solutions constructed by Elgindi may not be smooth, they are nonetheless in a class where there is strong local wellposedness, including local existence and continuous dependence on the initial data.

Axisymmetric, swirl-free solutions of the Euler equation are completely determined by a scalar vorticity $\omega_\theta$, where
\begin{equation}
    \omega(x)=\omega_\theta(r,z)e_\theta.
\end{equation}
This scalar vorticity satisfies the evolution equation
\begin{equation}
\partial_t\omega_\theta
+u\cdot\nabla\omega_\theta
-\frac{u_r}{r}\omega_\theta
=0.
\end{equation}
Here the vortex stretching term
$\frac{u_r}{r}\omega_\theta$ can produce growth, but this growth is depleted by the advection. In particular, it is simple to observe that $\frac{\omega_\theta}{r}$ is transported by the flow with
\begin{equation}
(\partial_t+u\cdot\nabla)\frac{\omega_\theta}{r}
=0.
\end{equation}
This implies that if $\frac{\omega_\theta}{r}$ is bounded in a suitable space, then there must be global smooth solutions. Danchin proved that $\frac{\omega^0_\theta}{r}\in 
L^{3,1}(r \diff r\diff z)$ and
$\omega^0_\theta\in 
L^{3,1}(r \diff r\diff z) \cap L^\infty$
is sufficient to guarantee a global solution \cite{Danchin}. Incidentally, these conditions automatically hold when
$u^0\in H^s, s>\frac{5}{2}$ (and $u^0$ is axisymmetric, swirl-free).

Elgindi is able to work around this in \cite{Elgindi}, because if $\omega_\theta \sim r^\alpha$ for some $\alpha \ll 1$, then the quantity $\frac{\omega_\theta}{r}$ is too singular for the fact that it is transported by the flow to prevent blowup. The result is that the smaller the value of $\alpha$ (i.e. the rougher the data while retaining the H\"older continuity of the vorticity), the weaker the effect of advection relative to vortex stretching,
and the more singular the dynamics.

One of the advantages of the axisymmetric, swirl-free formulation is that it can be described by a scalar evolution equation for the purely azimuthal vorticity. There is really no reason to think that the finite-time blowup described in \cites{Elgindi,ElgindiGhoulMasmoudi} is the most singular possible behaviour for solutions of the Euler equation in $\mathbb{R}^3$.
Rather, there is a scalar evolution equation in which the geometry of  singularity formation can be expressed very clearly in terms of sign conditions on $\omega_\theta$ in a way that is intractable for more complicated geometries.
The problem is that, because axisymmetry is a continuous symmetry, there is a reduction in dimension, and smooth, axisymmetric, swirl-free solutions of the three dimensional Euler equation have too much in common with two dimensional solutions of the Euler equation to blowup in finite-time.

For permutation symmetric solutions of the Euler equation, we will also be able reduce the vorticity equation to a scalar evolution equation, this time in terms of a single vorticity component. We will consider the evolution in terms of $\omega_1$, but to consider the dynamics of $\omega_2$ or $\omega_3$ would be equivalent. The important thing is that permutation symmetry guarantees that any one vorticity component completely determines both $\omega$ and $u$, giving us a scalar evolution equation that encodes all of the dynamics. At the same time, because the symmetry group is discrete, there is no reduction in dimension. The problem is still fully three dimensional, and finite-time blowup for smooth solutions is, at the very least, not ruled out.

\begin{theorem} 
\label{OneVortCompThmIntro}
    For all $\omega_1^0\in H^s_*\cap \dot{H}^{-1}, s>\frac{3}{2}$, there exists a unique solution 
    $\omega_1\in C\left([0,T_{max});
    H^s_*\cap\dot{H}^{-1}\right)
    \cap
    C^1\left([0,T_{max});
    H^{s-1}_*\right)$,
    to the single-component, permutation-symmetric  Euler vorticity equation 
    \begin{equation}
    \partial_t\omega_1
    +(u\cdot\nabla)\omega_1
    -(\omega\cdot\nabla)u_1
    =0,
    \end{equation}
    where
    \begin{align}
    \omega_2(x)
    &=
    -\omega_1(x_2,x_1,x_3) \\
    \omega_3(x)
    &=
    -\omega_1(x_3,x_2,x_1).
    \end{align}
    \begin{equation}
    u(x)=\int_{\mathbb{R}^3}
    G(x,y) \omega_1(y) \diff y,
    \end{equation}
    and
    \begin{equation}
    G(x,y)=
    \frac{1}{|x-y|^3}\left(\begin{array}{c}
         0  \\ -x_3+y_3 \\ x_2-y_2 
    \end{array}\right)
    +
    \frac{1}{|x-P_{12}(y)|^3}
    \left(\begin{array}{c}
         -x_3+y_3  \\ 0 \\ x_1-y_2 
    \end{array}\right)
    +
    \frac{1}{|x-P_{13}(y)|^3}
    \left(\begin{array}{c}
         x_2-y_2  \\ -x_1+y_3 \\ 0 
    \end{array}\right).
    \end{equation}
    Note that $P_{12}(x)=(x_2,x_1,x_3)$
and $P_{13}(x)=(x_3,x_2,x_1)$.
    
    We have a lower bound on the time of existence
    \begin{equation}
    T_{max}
    \geq 
    \frac{C_s}{\left\|\omega_1^0
    \right\|_{H^s\cap\dot{H}^{-1}}},
    \end{equation}
    where $C_s>0$ is an absolute constant independent of $u^0$.
    Furthermore, $u\in C\left([0,T_{max});
    H^{s+1}_{df}\right)
    \cap
    C^1\left([0,T_{max});
    H^{s}_{df}\right)$ is a permutation symmetric solution of the Euler equation.
    There are also permutation symmetric variants of the energy equality and the Beale-Kato-Madja Criterion:
    for all $0<t<T_{max}$, 
    \begin{equation}
    \|\omega_1(\cdot,t)
    \|_{\dot{H}^{-1}}^2
    =
    \left\|\omega_1^0
    \right\|_{\dot{H}^{-1}}^2;
    \end{equation}
    if $T_{max}<+\infty$,
    \begin{equation}
    \int_0^{T_{max}}
    \|\omega_1(\cdot,t)\|_{L^\infty}
    \diff t
    =+\infty.
    \end{equation}
\end{theorem}

\begin{remark}
    The constraint space $H^s_*\cap\dot{H}^{-1}$ is precisely the space of functions that are the first components of the curl of permutation symmetric vector fields---equivalently are the first component of divergence free, skew symmetric vector fields. A precise definition will be given in \Cref{OneCompConstraintDef}, and the equivalence of this definition to being the first component of the curl of a permutation symmetric vector field will be proven in \Cref{ConstraintThm}.
\end{remark}

\begin{remark}
The basis for this reduction is that all of the symmetries of a vector field can be expressed in terms of its vorticity. In general, for any $Q\in O(3)$, a vector field $u\in H^s_{df}$ is $Q$-symmetric, that is
\begin{equation}
    u=u^Q,
\end{equation}
if and only if
\begin{equation}
    \omega= \det(Q)\omega^Q.
\end{equation}
Note that for any vector field $v$,
\begin{equation}
v^Q(x)=Q v\left(Q^{tr}x\right)
\end{equation}
Using this result and the fact that $P_{12}$ and $P_{13}$ are generators of the permutation group $\mathcal{P}_3$, we will show that a vector field $u\in H^s_{df}$ is permutation symmetric if and only if
\begin{equation}
  \omega= -\omega^{P_{12}}
  = -\omega^{P_{13}}.
\end{equation}
This immediately allows the reduction of the Euler dynamics to one vorticity component, including a Biot-Savart law for the velocity in terms of one vorticity component.
\end{remark}

There are also permutation symmetric, $C^{1,\alpha}$ solutions of the Euler equation that blowup in finite-time.

\begin{theorem} \label{AlphaPermuteBlowupIntro}
    There exists $\alpha>0$ and an odd, permutation symmetric, $\sigma$-mirror symmetric solution of the incompressible Euler equation
    $u\in C\left([0,1);L^2\cap C^{1,\alpha}\right)
    \cap C^1\left([0,1);L^2\cap C^\alpha\right)$,
    satisfying
    \begin{equation}
    \int_0^1 \|\omega(\cdot,t)\|_{L^\infty}
    \diff t
    =+\infty,
    \end{equation}
    and for all $0\leq t <1$,
    \begin{equation}
    \|u(\cdot,t)\|_{L^2}^2
    =\left\|u^0\right\|_{L^2}^2.
    \end{equation}
\end{theorem}

\begin{remark}
    It should be noted that while it is very interesting that there are permutation symmetric solutions of the Euler equation in the class $C^{1,\alpha}$ that blowup in finite-time, these are not new blowup solutions of the Euler equation. 
    Rather we show that if we consider the $C^{1,\alpha}$ blowup soultions of the Euler equation developed by Elgindi \cite{Elgindi} and Elgindi, Ghoul, and Masmoudi \cite{ElgindiGhoulMasmoudi}, which are axisymmetric and swirl-free, and we take the axis of symmetric to be the axis $x_1=x_2=x_3$, rather than the $x_3$-axis, then these blowup solutions are permutation symmetric.
    What is interesting is that the permutation symmetric geometry is much less restrictive than axisymmetry, and so considering general permutation symmetric solutions may allow these arguments to apply to solutions with more regularity.
    Elgindi remarks in \cite{Elgindi} that: 
``We also remark, importantly, that while the methods
used here are applicable to axi-symmetric solutions without swirl, it is likely that they are also applicable in less rigid geometries and that in such settings one might be able to get much smoother solutions that develop singularities."
    The fact that, in the $C^{1,\alpha}$ case, Elgindi's arguments apply to the axisymmetric subspace of permutation symmetric solutions suggests that the permutation symmetric subspace without any axisymmetric condition is a potential candidate for this less rigid geometry. Furthermore, this is exactly the geometry in which the author proved finite-time blowup for smooth solutions of the Fourier-restricted Euler equation in \cite{MillerRestricted}, and these solutions did not posses any continuous symmetry.
\end{remark}

While we now have an evolution equation in terms of $\omega_1$, it is not immediately clear what our sign condition ought to be.
In the case of axisymmetric, swirl-free flows, the sign condition requires taking $\omega_\theta$ odd in $z$, with $\omega_\theta(r,z)\leq 0$ for $z\geq 0$ and $\omega_\theta(r,z)\geq 0$ for $z\leq 0$. Physically, this corresponds to two colliding, mirror symmetric, vortex rings.
In this case we find that the strain matrix $S_{ij}=\frac{1}{2}(\partial_iu_j+\partial_ju_i)$ has the structure
\begin{equation}
 S(0)=\lambda 
 \left(\begin{array}{ccc}
     1 & 0 & 0  \\
     0 &  1 & 0 \\
     0 & 0 &  -2
 \end{array}\right),
\end{equation}
where $\lambda>0$.
Note that this strain has two equal positive eigenvalues, and one negative eigenvalue. 
In work on the incompressible Navier--Stokes equation, Neustupa and Penel proved a scale critical regularity criterion on the positive part of the middle eigenvalue of the strain matrix \cites{NeustupaPenel1,NeustupaPenel2}.
The worst case of this regularity criterion is when $\lambda_2=\lambda_3>0$ and $\lambda_1=-2\lambda_3$.
These results, which were also revisited by the author in \cite{MillerStrain}, follow from an identity for enstrophy growth. For smooth solutions of the Euler equation, this identity has the form
\begin{equation}
\frac{\diff}{\diff t}\|S(\cdot,t)\|_{L^2}^2
=-4\int\det(S).
\end{equation}

In order to find the relevant sign condition for $\omega_1$, we will use the Biot-Savart law to compute the strain at the origin, finding that
\begin{equation} \label{StrainEqnIntro}
    S(0)
    =
    \lambda 
    \left( \begin{array}{ccc}
        0 & -1 & -1  \\
        -1 & 0 & -1  \\
        -1 & -1 & 0
    \end{array}\right),
    \end{equation}
    where
    \begin{equation}
        \lambda
    = \label{LambdaEqnIntro}
    \frac{3}{8\pi}\int_{\mathbb{R}^3}
    \left(\frac{(x_1+x_2+x_3)(x_2-x_3)}
    {|x|^5}\right)\omega_1(x) \diff x.
    \end{equation}
This matrix has the exact same $-2\lambda,\lambda,\lambda$ eigenvalue structure as the colliding vortex rings in the axisymmetric, swirl-free setting. In this case, the axis $x_1=x_2=x_3$ corresponds to the eigenvalue $-2\lambda$ and the plane $x_1+x_2+x_3=0$ corresponds to the (multiplicity two) eigenvalue $\lambda$.
Permutation symmetry guarantees that $\omega(0)=0$, and if we take an odd velocity field $u$, then we will have $u(0)=0$.
Consequently, when $\lambda>0$, there will be the possibility of a stagnation point blowup at the origin with unbounded planar stretching and axial compression.

\begin{remark}
    The identity for $\lambda$ will provide the sign condition that we need to identify singular vorticity candidates in the permutation symmetric setting.
    In order to generate as much planar stretching (and consequently vortex stretching) near the origin as possible, 
    the sign condition that we impose is that we want $\omega_1$ to have the same sign as 
    $(x_1+x_2+x_3)(x_2-x_3)$. As we will see later on, this can be achieved globally in space; however, our most natural examples will be axisymmetric, swirl-free, with the $x_1=x_2=x_3$ axis serving as the axis of symmetry, and this is a setting where the finite-time blowup of smooth solutions of the Euler equation is ruled out.
\end{remark}

Because vector fields that are odd, permutation symmetric, and mirror symmetric about the plane $x_1+x_2+x_3=0$ play such an important role in our analysis, we will refer to these vector fields as $\mathcal{G}_\sigma$-symmetric. Here the axis $x_1=x_2=x_3$ serves as the symmetry axis, but it is also worth considering what happens when this symmetry group is conjugated by a rotation that maps that axis $x_1=x_2=x_3$ to the $x_3$ axis.
We will say that a vector field $u$ is $\mathcal{G}$-symmetric if it is mirror symmetric in each of the coordinate axes and symmetric with respect to rotations by $\frac{\pi}{3}$ in the horizontal plane. 
More precise definitions of the $\mathcal{G}$ and $\mathcal{G}_\sigma$ symmetry groups will be given in the next subsection.
These two symmetries are in fact equivalent up to rotation.

\begin{theorem} \label{EquivThmIntro}
    Suppose $u\in C\left(\mathbb{R}^3;
    \mathbb{R}^3\right)$.
    Then $u$ is $\mathcal{G}$-symmetric if and only if $u^{Q_\sigma}$ is 
    $\mathcal{G}_\sigma$-symmetric,
    where
    \begin{equation}
    Q_\sigma
    =
    \left(\begin{array}{ccc}
\frac{1}{\sqrt{2}} & \frac{1}{\sqrt{6}} 
& \frac{1}{\sqrt{3}} \\
-\frac{1}{\sqrt{2}} & \frac{1}{\sqrt{6}} 
& \frac{1}{\sqrt{3}} \\
0 & -\frac{2}{\sqrt{6}} 
& \frac{1}{\sqrt{3}} \\
    \end{array}\right),
    \end{equation}
and
\begin{equation}
    u^{Q_\sigma}(x)
    =
    Q_\sigma
    u\left(Q_\sigma^{tr}x\right).
\end{equation}
\end{theorem}

There is an advantage to considering the symmetry group $\mathcal{G}$, because $\mathcal{G}$-symmetric, divergence free vector fields have a very natural expression as a Fourier series in cylindrical coordinates, with  Fourier modes in the $\theta$ variable in $6\mathbb{Z}^+$. In particular, we can see that axisymmetric, swirl-free vector fields with a mirror symmetry in the $z$ direction are a subspace of $\mathcal{G}$-symmetric vector fields from the leading order term in this Fourier series.

\begin{theorem} \label{FourierSeriesIntro}
    Suppose $u\in H^s\left(\mathbb{R}^3;
    \mathbb{R}^3\right), s>\frac{5}{2}$.
    Then $u$ is $\mathcal{G}$-symmetric and divergence free if and only if it can be expressed as a Fourier series in cylindrical coordinates by
    \begin{equation}
    u(x)=
    u_{r,0}(r,z)e_r
    +u_{z,0}(r,z)e_z
    +\sum_{n=1}^{\infty}
    u_{r,n}(r,z)\cos(6n\theta)e_r
    +u_{z,n}(r,z)\cos(6n\theta)e_z
    +u_{\theta,n}(r,z)
    \sin(6n\theta)e_\theta,
    \end{equation}
    where for all $n\in\mathbb{Z}^+$, $u_{r,n},u_{\theta,n}$ are even in $z$
    and $u_{z,n}$ is odd in $z$,
    and furthermore, due to the divergence free constraint,
    \begin{equation}
    \partial_r u_{r,0}
    +\frac{1}{r}u_{r,0}
    +\partial_z u_{z,0}
    =0,
    \end{equation}
    and for all $n\in\mathbb{N}$,
    \begin{equation}
    \partial_r u_{r,n}
    +\frac{1}{r}u_{r,n}
    +\partial_z u_{z,n}
    +\frac{6n}{r}u_{\theta,n}
    =0.
    \end{equation}
    Note that this means that $u_{\theta,n}$ is completely determined by $u_{r,n}, u_{z,n}$ with
    \begin{align}
    u_{\theta,n}
    &=
    -\frac{r}{6n}
    \left(\partial_r u_{r,n}
    +\frac{1}{r}u_{r,n}
    +\partial_z u_{z,n}\right) \\
    &=
    -\frac{1}{6n}
    \left(\partial_r(ru_{r,n})
    +\partial_z(ru_{z,n})\right),
    \end{align}
    and that $u_{\theta,0}=0$ by convention.
\end{theorem}

Because $\mathcal{G}$-symmetry is preserved by the dynamics of the Euler equation, that means that if we have any initial data of the form in \Cref{FourierSeriesIntro}, this will be preserved by the dynamics of the Euler equation, giving us an Anstaz to study the dynamics.

\begin{corollary} \label{FourierSeriesCorIntro}
    Suppose $u\in C\left([0,T_{max});H^s_{df}\right), s>\frac{5}{2}$ is a solution of the Euler with $\mathcal{G}$-symmetric initial data $u^0\in H^s_{df}$. Then  $u(\cdot,t)$ is $\mathcal{G}$-symmetric for all $0\leq t<T_{max}$; consequently, for all $0\leq t<T_{max}$,
    \begin{multline}
    u(x,t)=
    u_{r,0}(r,z,t)e_r
    +u_{z,0}(r,z,t)e_z \\
    +\sum_{n=1}^{\infty}
    u_{r,n}(r,z,t)\cos(6n\theta)e_r
    +u_{z,n}(r,z,t)\cos(6n\theta)e_z
    +u_{\theta,n}(r,z,t)
    \sin(6n\theta)e_\theta,
    \end{multline}
    where for all $n\in\mathbb{Z}^+,$ 
    $u_{r,n}, u_{\theta,n}$ are even in $z$
    and $u_{z,n}$ is odd in $z$,
    and furthermore
    \begin{equation}
    \partial_r u_{r,0}
    +\frac{1}{r}u_{r,0}
    +\partial_z u_{z,0}
    =0,
    \end{equation}
    and for all $n\in\mathbb{N}$,
    \begin{equation}
    \partial_r u_{r,n}
    +\frac{1}{r}u_{r,n}
    +\partial_z u_{z,n}
    +\frac{6n}{r}u_{\theta,n}
    =0.
    \end{equation}
\end{corollary}

\begin{remark}
    Elgindi and Jeong previously considered solutions of the Euler equation under a group of discrete symmetries related to the groups $\mathcal{G}$ and $\mathcal{G}_\sigma$,
    studying solutions that are permutation symmetric and mirror symmetric about each of the coordinate axes \cite{ElgindiJeongOctahedral}.
    They proved finite-time blowup for weak solutions of the Euler equation with a vorticity that is bounded and piece-wise H\"older continuous, but with jump discontinuities at the symmetry planes.
    In order for these results to apply to smooth or even $C^{1,\alpha}$ solutions of the Euler equation, additional assumptions would need to be imposed on the vorticity at the symmetry planes (i.e. the vanishing of the non-normal vorticity), as the result is based on constructing a vorticity in the whole space based on reflections of the vorticity from an octahedral domain.

    The advantage of the discrete symmetry groups $\mathcal{G}$ and $\mathcal{G}_\sigma$ over the discrete symmetry groups considered by Elgindi and Jeong in \cite{ElgindiJeongOctahedral}, is that the strain matrix at the origin automatically has a $-2\lambda,\lambda,\lambda$ eigenvalue structure at the origin, and so it is only necessary to maintain the sign condition $\lambda>0$ to obtain the most singular scenario predicted by Neustupa and Penel \cite{NeustupaPenel1}. The gradient at the origin must have the form
    \begin{align}
    \nabla u(\Vec{0}) &=
    \lambda\left(\begin{array}{ccc}
         0 & -1 & -1  \\
         -1 & 0 & -1 \\
         -1 & -1 & 0
    \end{array}\right) \\
    \nabla u(\Vec{0})  &=
    \lambda\left(\begin{array}{ccc}
         1 & 0 & 0  \\
         0 & 1 & 0 \\
         0 & 0 & -2
    \end{array}\right),
    \end{align}
    for $\mathcal{G}_\sigma$-symmetric and $\mathcal{G}$-symmetric solutions of the Euler equation respectively.
    On the contrary, smooth solutions of the Euler equation with the symmetries considered in \cite{ElgindiJeongOctahedral} must satisfy
    \begin{equation}
    \nabla u (\Vec{0}) =0,
    \end{equation}
    which means there will not be planar stretching/axial compression at the origin. It should also be noted that for $\mathcal{G}_\sigma$-symmetric solutions, the condition $\lambda>0$ provides a heuristic in searching for vorticities that generate singular dynamics based on \cref{LambdaEqnIntro}, which is not possible when the gradient vanishes at the origin, as must happen for smooth solutions with the discrete symmetries considered in \cite{ElgindiJeongOctahedral}.
\end{remark}

\begin{remark}
    There are some very interesting connections between solutions with the symmetry group $\mathcal{G}_\sigma$ considered in this paper and the very recent numerical study of potential Euler singularities by Protas and Zhao \cite{ProtasZhao}. Protas and Zhao studied the most singular dynamics of the three dimensional Euler equation on the torus  with a two step process: they solved the Euler equation numerically using pseudospectral methods, and they searched for initial data leading to the maximal amplification of the $H^3$ norm using a Riemannian conjugate gradient method. The optimization problem is nonconvex, and so there is the possibility of non-unique local maximizers. Nonetheless, for three of the four initial guesses considered, the Taylor-Green vortex, a random initial condition, and the Hou pure swirl initial condition, the same maximizer was obtained up to rotation and translation.

    The aspect that is the most interesting from our point of view is that the solution which exhibited maximum growth in \cite{ProtasZhao} has all of the symmetries of the group $\mathcal{G}_\sigma$, at least approximately. Protas and Zhao describe the geometry of the solution exhibiting the most singular behaviour as follows.
    ``We observe that throughout its evolution the
optimal flow has the form of two jets colliding head-on with the vorticity concentrating
into two strongly flattened vortex rings as time goes on. The region with large values
of $\log_{10}\left(|D|^3 u\right)$,
which is the quantity that matters in our objective functional,
evolves into a flat disc located in the middle of the two rings.
The vorticity field at the final time $t=75$ has three symmetry planes: $x_1=x_2, x_1=x_3$ and $x_2=x_3$, in addition to discrete rotation symmetries with respect to the body diagonal passing through the center of the two rings.'' 
Note that the vorticity having symmetry planes $x_1=x_2, x_1=x_3, x_2=x_3$ is a necessary and sufficient condition for a solution to be permutation symmetric, as we will see in \Cref{PermuteGenerateCor}.
The most singular solution considered in \cite{ProtasZhao} also exhibits planar stretching and axial compression at origin, just like that observed for our permutation symmetric solutions in \cref{StrainEqnIntro,LambdaEqnIntro},
and has a mirror symmetry about the plane perpendicular to the colliding jets, which means their most singular solution is also at least approximately $\sigma$-mirror symmetric.
Furthermore, a scenario in which there are colliding vortex rings, not quite axisymmetric but with a discrete rotational symmetry, is precisely the structure of the Ansatz for blowup suggested by \Cref{FourierSeriesCorIntro}.

What is the most remarkable about this fact is that these symmetries are in no way assumed by the numerical study in \cite{ProtasZhao}. The initial guesses for the most singular behaviour do not have these symmetries; it is only after an iteration scheme that searches for the most singular possible dynamics has converged that a solution with these symmetries is found. This suggests that the $\mathcal{G}_\sigma$-symmetric space of solutions of the Euler equation deserves further study for the finite-time blowup problem.
It should also be noted, however, that the most singular solution found in \cite{ProtasZhao} is not exactly permutation symmetric. For the initial data found to lead to the most singular dynamics, Protas and Zhao compute that
\begin{equation}
\|\omega_1\|_{L^2} \approx \|\omega_2\|_{L^2}
< \|\omega_3\|_{L^2},
\end{equation}
whereas for an exactly permutation symmetric solution these norms must all be equal, as we will see in \Cref{VortNormProp}. It would be interesting to apply the methods from \cite{ProtasZhao}, while restricting to exactly $\mathcal{G}_\sigma$-symmetric solutions,
to see if this geometric constraint would allow for more insight into the potential singularity formation. This could allow us to numerically explore the following question.
\end{remark}

\begin{question}
    Does there exist a $\mathcal{G}_\sigma$-symmetric, smooth solution of the incompressible Euler equation $u\in C^1\left([0,1);H^\infty_{df}\right)$ 
    that blows up in finite-time with
    \begin{equation}
    \int_0^1 \|\omega(\cdot,t)\|_{L^\infty}
    \diff t
    =+\infty?
    \end{equation}
\end{question}

\subsection{Definitions and notation}

In this section, we will give a number of key definitions of spaces and symmetries. We begin by defining the orthogonal group and the group of permutation matrices, along with permutation symmetry and some elementary properties.

\begin{definition}
    A $3\times 3$ matrix $Q\in O(3)$, the orthogonal group, if 
    \begin{equation}
    Q^{tr}Q=I_3.
    \end{equation}
    Note that this implies that the columns of $Q$,
    \begin{equation}
    v_i=Qe_i,
    \end{equation}
    form an orthonormal basis for $\mathbb{R}^3$.
\end{definition}

\begin{definition}
    We will take $\mathcal{P}_3\subset O(3)$
    to be the group of permutation matrices.
    $\mathcal{P}_3$ has six elements, the identity matrix
    \begin{equation}
    I_3=\left(\begin{array}{ccc}
         1 & 0 & 0  \\
         0 & 1 & 0  \\
         0 & 0 & 1
    \end{array}\right),
    \end{equation}
    three swap permutations
    \begin{align}
    P_{12}&=
    \left(\begin{array}{ccc}
         0 & 1 & 0  \\
         1 & 0 & 0  \\
         0 & 0 & 1
    \end{array}\right) \\
    P_{13}&=
    \left(\begin{array}{ccc}
         0 & 0 & 1  \\
         0 & 1 & 0  \\
         1 & 0 & 0
    \end{array}\right) \\
    P_{23}&=
    \left(\begin{array}{ccc}
         1 & 0 & 0  \\
         0 & 0 & 1  \\
         0 & 1 & 0
    \end{array}\right),
    \end{align}
    and two other permutations
    \begin{align}
     P_f&=
    \left(\begin{array}{ccc}
         0 & 0 & 1  \\
         1 & 0 & 0  \\
         0 & 1 & 0
    \end{array}\right) \\
    P_b&=
    \left(\begin{array}{ccc}
         0 & 1 & 0  \\
         0 & 0 & 1  \\
         1 & 0 & 0
    \end{array}\right)
    \end{align}
\end{definition}

\begin{definition}
    For all $u\in C\left(\mathbb{R}^3;
    \mathbb{R}^3\right)$
    and for all $Q\in O(3)$, define $u^Q$ by
    \begin{equation}
    u^Q(x)=Q u\left(Q^{tr}x\right).
    \end{equation}
\end{definition}

\begin{proposition} \label{OrthoComposeProp}
    For all $u\in C\left(\mathbb{R}^3;
    \mathbb{R}^3\right)$
    and for all $Q,Q'\in O(3)$,
    \begin{equation}
    \left(u^Q\right)^{Q'}
    =
    u^{Q'Q}.
    \end{equation}
    In particular, for all $Q\in O(3)$,
    \begin{equation}
    \left(u^Q\right)^{Q^{tr}}=u
    \end{equation}
\end{proposition}
\begin{proof}
    This is a linear algebra exercise left to the reader.
\end{proof}

\begin{definition}
    Let $Q\in O(3)$ be an orthogonal matrix. We will say that $Q$ is a symmetry of $u$ if
    \begin{equation}
    u^Q=u.
    \end{equation}
    We will also refer to $u$ as $Q$-symmetric in this case.
\end{definition}

\begin{definition} \label{PermSymDef}
    We will say that a vector field $u\in C\left(\mathbb{R}^3;\mathbb{R}^3\right)$
    is permutation symmetric if for all $P\in\mathcal{P}_3$,
    \begin{equation}
    u^P=u.
    \end{equation}
    Note that this is also a valid definition for vector fields with less regularity. For $u\in L^p\left(\mathbb{R}^3;
    \mathbb{R}^3\right)$,
    this definition of permutation symmetry still holds,
    but in this case the equality is almost everywhere $x\in\mathbb{R}^3$.
\end{definition}

\begin{definition}
    We will say that a vector field $w\in C\left(\mathbb{R}^3;
    \mathbb{R}^3\right)$ is permutation skew-symmetric, if
    \begin{equation}
    w=w^{P_f}=w^{P_b}
    =-w^{P_{12}}
    =-w^{P_{13}}
    =-w^{P_{23}}.
    \end{equation}
    Note that this condition can be stated equivalently as
    \begin{equation}
    w=\det(P)w^P,
    \end{equation}
    for all $P\in\mathcal{P}_3$.
\end{definition}

\begin{remark}
    We introduce the notion of permutation skew symmetric vector fields, because will will prove that a vector field $u$ satisfying. the divergence free constraint $\nabla\cdot u=0$ is permutation symmetric if and only if the associated vorticity $\omega=\nabla\times u$ is permutation skew symmetric.
\end{remark}

An important vector in the study of permutation symmetry is 
\begin{equation}
\sigma
=
\left(\begin{array}{c}
      1 \\ 1 \\ 1 
\end{array}\right),
\end{equation}
because the axis $\spn(\sigma)$ is the set of points that are invariant under all permutations. We will define $\Tilde{\sigma}$ to be the unit vector in this direction
\begin{equation}
\Tilde{\sigma}=\frac{1}{\sqrt{3}}\sigma.
\end{equation}

\begin{definition} \label{MirrorSymDef}
    For any unit vector $v\in \mathbb{R}^3, |v|=1$,
    we will define the reflection matrix $M_v$ to be
    \begin{equation}
    M_v=I_3- 2v\otimes v.
    \end{equation}
    Furthermore, we will say that a vector field $u\in C\left(\mathbb{R}^3;\mathbb{R}^3\right)$ is $v$-mirror symmetric if
    \begin{equation}
    u^{M_v}=u.
    \end{equation}
    Even though $\sigma$ is not a unit vector, we will say that a vector field is $\sigma$-mirror symmetric, if and only if it is $\Tilde{\sigma}$-mirror symmetric.
\end{definition}

\begin{definition}
    We will denote rotations in the horizontal plane by an angle $\theta$ as $R_\theta$, with
    \begin{equation}
    R_\theta
    =
    \left(\begin{array}{ccc}
        \cos(\theta) & -\sin(\theta) & 0 \\
         \sin(\theta) & \cos(\theta) & 0 \\
         0 & 0 & 1
    \end{array}\right).
    \end{equation}
\end{definition}

\begin{definition}
    Let $\mathcal{G} \subset O(3)$ be the group generated by $M_{e_1},M_{e_2},M_{e_3},
    R_{\frac{\pi}{3}}$.
\end{definition}

\begin{definition}
    Let $\mathcal{G}_\sigma \subset O(3)$ be the group generated by 
    $-I_3, M_{\Tilde{\sigma}},
    \mathcal{P}_3$.
\end{definition}

\begin{definition}
    For all $n\in\mathbb{N}$, let $\mathcal{R}_n$ be the group generated by $R_\frac{2\pi}{n}$,
    \begin{equation}
    \mathcal{R}_n
    =\left\{R_\frac{2\pi j}{n},
    0\leq j \leq n-1
    \right\}.
    \end{equation}
\end{definition}

\begin{definition}
    We will say that a vector field $u\in C\left(\mathbb{R}^3;\mathbb{R}^3\right)$,
    is odd if for all $x\in\mathbb{R}^3$
    \begin{equation}
        u(-x)=-u(x).
    \end{equation}
    We will say that $u$ is component-wise odd if for all $i\in\{1,2,3\}, u_i$ is odd in $x_i$ and even in $x_j$ for all $j\neq i$. 
\end{definition}

    \begin{remark}
    Note that every component-wise odd vector field is odd, but the converse does not hold. Also note that every $\mathcal{G}$-symmetric vector field is component-wise odd, because a vector field being component-wise odd is equivalent to being mirror symmetric with respect to the cardinal directions $e_1,e_2,e_3$.
    \end{remark}

\begin{remark}
    We will see that $\mathcal{G}$ and $\mathcal{G}_{\sigma}$ both have 24 elements, where
    \begin{align}
    \mathcal{G}
    &=
    \pm \mathcal{R}_3
    \cup
    \pm M_{e_1}\mathcal{R}_3
    \cup 
    \pm M_{e_3}\mathcal{R}_3
    \cup 
    \pm M_{e_3} M_{e_1}\mathcal{R}_3 \\
    \mathcal{G}_\sigma
    &=
    \pm \mathcal{P}_3 \cup 
    \pm M_{\Tilde{\sigma}} \mathcal{P}_3.
    \end{align}
\end{remark}

We now define homogeneous and inhomogeneous Sobolev spaces.

\begin{definition}
    For all $s\in\mathbb{R}$, we will take $H^s\left(\mathbb{R}^3\right)$ to be the Hilbert space with the norm
    \begin{equation}
    \|f\|_{H^s}=
    \left(\int_{\mathbb{R}^3}
    \left(1+4\pi^2|\xi|^2\right)^s
    \left|\hat{f}(\xi)\right|^2
    \diff\xi\right)^\frac{1}{2}.
    \end{equation}
    For all $-\frac{3}{2}<s<\frac{3}{2}$, we will take $\dot{H}^s\left(\mathbb{R}^3\right)$
    to be the Hilbert space with norm
    \begin{equation}
    \|f\|_{\dot{H}^s}=
    \left(\int_{\mathbb{R}^3}
    \left(4\pi^2|\xi|^2\right)^s
    \left|\hat{f}(\xi)\right|^2
    \diff\xi\right)^\frac{1}{2}.
    \end{equation}
    These two definitions give us a definition of the space $H^s\cap \dot{H}^{-1}$ for all $s>-1$, but we will nonetheless define the norm of this Hilbert space as follows:
    \begin{equation}
    \|f\|_{H^s\cap\dot{H}^{-1}}=
    \left(\int_{\mathbb{R}^3}
    \frac{\left(1+4\pi^2|\xi|^2\right)^{s+1}}{4\pi^2|\xi|^2}
    \left|\hat{f}(\xi)\right|^2
    \diff\xi\right)^\frac{1}{2}.
    \end{equation}
\end{definition}

\begin{remark}
Note that we use this definition for $H^s\cap \dot{H}^{-1}$, rather than the standard norm for the intersection of Hilbert spaces $\mathcal{H}_a\cap\mathcal{H}_b$
\begin{equation}
\|f\|_{\mathcal{H}_a\cap\mathcal{H}_b}
=\left(\|f\|_{\mathcal{H}_a}^2
+\|f\|_{\mathcal{H}_b}^2
\right)^\frac{1}{2},
\end{equation}
because $\nabla:H^{s+1} \to H^s\cap \dot{H}^{-1}$, and we want the isometry
\begin{equation}
\|\nabla f\|_{H^s\cap \dot{H}^{-1}}
=\|f\|_{H^{s+1}}.
\end{equation}
\end{remark}

\begin{definition}
    We define the the space of divergence free vector fields in $H^s$ by enforcing the constraint $\nabla\cdot u=0$ pointwise in Fourier space.
    Suppose $u\in H^s\left(\mathbb{R}^3;
    \mathbb{R}^3\right), s\geq 0$. Then $u\in H^s_{df}$ if and only if
    \begin{equation}
    \xi\cdot\hat{u}(\xi)=0,
    \end{equation}
    almost everywhere $\xi\in\mathbb{R}^3$.
\end{definition}

\begin{definition} \label{OneCompConstraintDef}
    For all $s \geq 0$, we will define the constraint space $H_*^s\cap\dot{H}^{-1}$ as follows;
    suppose $f\in H^s\cap\dot{H}^{-1}$.
    Then $f\in H_*^s\cap\dot{H}^{-1}$ if and only if for all $x\in\mathbb{R}^3$
    \begin{equation}
    f(x)=-f(P_{23}x),
    \end{equation}
    and for all $\xi\in\mathbb{R}^3$
    \begin{equation}
    \xi_1 \hat{f}(\xi) 
    -\xi_2 \hat{f}(P_{12}\xi)
    -\xi_3 \hat{f}(P_{13}\xi)
    =0.
    \end{equation}
\end{definition}

\begin{remark}
    Note that the above equality in Fourier space must hold almost everywhere $\xi\in\mathbb{R}^3$, and will hold pointwise if $f\in L^1$, while continuity means that the physical space constraint will hold everywhere.
    This constraint space is important, because it is precisely the space of functions $\omega_1$ that are the first components of the curl of permutation-symmetric, divergence-free vector fields. If we wish to study the the dynamics of permutation-symmetric vector fields by considering only the evolution equation for the first vorticity component, this is the relevant constraint space.
\end{remark}

\begin{definition}
    For all $u\in H^s_{df}, s>\frac{5}{2}$, we will take $A$ and $S$ to be the symmetric and anti-symmetric part of the gradient respectively, with
    \begin{align}
    A_{ij}&=\frac{1}{2}
    (\partial_i u_j-\partial_j u_i) \\
    S_{ij}&=\frac{1}{2}
    (\partial_i u_j+\partial_j u_i).
    \end{align}
    Note that the the vorticity $\omega$ is related to $A$ by
    \begin{equation} \label{VortMatrix}
    A=\frac{1}{2}\left
    (\begin{array}{ccc}
       0  & \omega_3 & -\omega_2  \\
        -\omega_3 &  0 & \omega_1 \\
        \omega_2 & -\omega_1 & 0
    \end{array}\right).
    \end{equation}
\end{definition}

\section{Vorticity and permutation symmetry}

We have defined permutation symmetry in terms of the velocity. In this section, we will consider how the vorticity behaves with respect to permutation symmetry. We will then use these results to understand the dynamics of permutation symmetric solutions of the Euler equation in terms of a single vorticity component, proving \Cref{OneVortCompThmIntro}.

\subsection{The vorticity under orthogonal transformations}

The first step, will be to understand how the vorticity behaves with respect to general orthogonal transformations.

\begin{proposition} \label{GradProp}
    For all $u\in C^1\left(\mathbb{R}^3;
    \mathbb{R}^3\right)$,
    \begin{equation}
    \nabla u^Q(x)=Q \nabla u\left(Q^{tr}x\right) Q^{tr}.
    \end{equation}
\end{proposition}

\begin{proof}
    This computation is classical, but is included for completeness.
    Recall that 
    \begin{equation}
    u^Q(x)=Qu\left(Q^{tr}x\right),
    \end{equation}
    and so
    \begin{equation}
    u^Q_j(x)=\sum_{k=1}^3 Q_{jk}u_k(Q^{tr}x).
    \end{equation}
    The chain rule then implies that
    \begin{align}
    \partial_i u^Q_j(x)
    &=
    \sum_{k,m=1}^3 
    Q_{jk} (\partial_m u_k)(Q^{tr}x) Q^{tr}_{mi} \\
    &=
    \sum_{k,m=1}^3 
    Q_{jk} Q_{im} (\nabla u)_{mk}(Q^{tr}x) \\
    &=
    \left(Q \nabla u(Q^{tr}x) Q^{tr}\right)_{ij},
    \end{align}
    and this completes the proof.
\end{proof}

\begin{proposition} \label{O3prop}
    For all $Q\in O(3)$
    \begin{equation}
    \det(Q)=\pm 1,
    \end{equation}
    and furthermore
    \begin{align}
    v_1\times v_2&=\det(Q) v_3 \\
    v_2\times v_3&=\det(Q) v_1 \\
    v_3\times v_1&=\det(Q) v_2,
    \end{align}
    where $v_i=Qe_i$.
\end{proposition}

\begin{proof}
    If $Q\in O(3)$, then the columns of $Q$ form an orthonormal basis for $\mathbb{R}^3$.
    This implies that $v_1$ and $v_2$ are othogonal unit vectors, and furthermore that
    \begin{align}
    |v_1\times v_2|&=1 \\
    v_1\times v_2 &\in \spn{v_3}.
    \end{align}
    Therefore we can see that
    \begin{equation}
    v_1\times v_2=\pm v_3.
    \end{equation}
    We know that 
    \begin{equation}
    \det(Q)=
    (v_1\times v_2)\cdot v_3,
    \end{equation}
    and so we can conclude that
    \begin{equation}
    \det(Q)=\pm 1,
    \end{equation}
    and that 
    \begin{equation}
    v_1\times v_2
    =\det(Q) v_3.
    \end{equation}
    The proofs of the identities for $v_2\times v_3$ and $v_3\times v_1$ are entirely analogous.
\end{proof}

\begin{theorem} \label{VortThmSO}
    For all $u\in C^1\left(\mathbb{R}^3;
    \mathbb{R}^3\right)$,
    \begin{equation}
    \nabla\times u^Q
    =
    \det(Q) \omega^Q.
    \end{equation}
\end{theorem}

\begin{proof}
    We can see from \Cref{GradProp} that 
    \begin{equation}
    A^Q(x)=Q A\left(Q^{tr}x\right) Q^{tr},
    \end{equation}
    where $A^Q$ is the anti-symmetric part of $\nabla u^Q$.
    We need to show that for all $v\in \mathbb{R}^3$
    \begin{equation}
    A^Q(x) v=
    \frac{1}{2}\det(Q) Q\omega(Q^{tr}x)\times v.
    \end{equation}
    Because $Q\in O(3)$, we can see that 
    $\left\{v_i=Qe_i\right\}_{i=1}^3$ is an orthonormal basis for $\mathbb{R}^3$, and so therefore it suffices to show that for all $1\leq i \leq 3$
    \begin{equation}
    A^Q(x) v_i=
    \frac{1}{2}\det(Q) Q\omega(Q^{tr}x)\times v_i.
    \end{equation}

    We will begin with the computation for $v_1$
    We can compute that
    \begin{align}
    A^Q(x)v_1
    &=
    Q A\left(Q^{tr}x\right) Q^{tr}v_1 \\
    &=
    Q A\left(Q^{tr}x\right) e_1 \\
    &=
    \frac{1}{2} Q \omega\left(Q^{tr}x\right)\times e_1 \\
    &=
    \frac{1}{2} Q \left(\begin{array}{c}
         0 \\ \omega_3 \\ -\omega_2
    \end{array}\right)\left(Q^{tr}x\right) \\
    &=
    \frac{1}{2}\omega_3\left(Q^{tr}x\right)v_2
    -\frac{1}{2}\omega_2\left(Q^{tr}x\right)v_3.
    \end{align}
    Applying \Cref{O3prop}, we compute that
    \begin{align}
    \det(Q) Q\omega\left(Q^{tr}x\right)\times v_1
    &=
    \det(Q) \left(\omega_1\left(Q^{tr}x\right)v_1
    +\omega_2\left(Q^{tr}x\right)v_2
    +\omega_3\left(Q^{tr}x\right)v_3\right)
    \times v_1\\
    &=
    \det(Q)\omega_2\left(Q^{tr}x\right)v_2\times v_1
    +\det(Q)\omega_3\left(Q^{tr}x\right)v_3\times v_1 \\
    &=
    -\omega_2\left(Q^{tr}x\right) v_3
    +\omega_3\left(Q^{tr}x\right)v_2.
    \end{align}
    Therefore, we may conclude that
    \begin{equation}
    A^Q(x) v_1=
    \frac{1}{2}\det(Q) Q\omega(Q^{tr}x)\times v_1.
    \end{equation}
    The proof that
    \begin{align}
    A^Q(x) v_2
    &=
    \frac{1}{2}\det(Q) Q\omega(Q^{tr}x)\times v_2 \\
    A^Q(x) v_3
    &=
    \frac{1}{2}\det(Q) Q\omega(Q^{tr}x)\times v_3,
    \end{align}
    is entirely analogous and left to the reader.
    We know from \eqref{VortMatrix} that $\omega^Q$ is the unique vector that satisfies 
    \begin{equation}
    A^Q v=\frac{1}{2} \omega^Q\times v.
    \end{equation}
    By hypothesis $Q\in O(3)$, and so 
    $\{v_1,v_2,v_3\}$ is an orthonormal basis for $\mathbb{R}^3$, and consequently for all $\mathbb{v}\in\mathbb{R}^3$,
    \begin{equation}
    A^Q(x) v=
    \frac{1}{2}\det(Q) Q\omega(Q^{tr}x)\times v,
    \end{equation}
    and we may conclude that
    \begin{equation}
    \omega^Q(x)=
    \det(Q) Q\omega(Q^{tr}x).
    \end{equation}
\end{proof}

\begin{remark}
    This result is classical, but is included for the sake of completeness.
    In physical terms, \Cref{VortThmSO} says that right handed coordinate transformations ($\det(Q)=1$) preserve the orientation of the vorticity, while left handed coordinate transformations ($\det(Q)=-1$) flip the orientation of the vorticity.
\end{remark}

\begin{corollary} \label{SwapPermuteCor}
    For all $u\in C^1\left(\mathbb{R}^3;
    \mathbb{R}^3\right)$,
    \begin{align}
    \nabla\times u^{P_{12}}
    &=-\omega^{P_{12}} \\
    \nabla\times u^{P_{13}}
    &=-\omega^{P_{13}} \\
    \nabla\times u^{P_{23}}
    &=-\omega^{P_{23}}
    \end{align}
\end{corollary}

\begin{proof}
    Observe that 
    \begin{equation}
    \det(P_{12})=\det(P_{13)}=\det(P_{23})=-1,
    \end{equation}
    and the result follows immediately from \Cref{VortThmSO}.
\end{proof}

\begin{corollary} \label{OtherPermuteCor}
    For all $u\in C^1\left(\mathbb{R}^3;
    \mathbb{R}^3\right)$,
    \begin{align}
    \nabla\times u^{P_f}
    &=\omega^{P_f} \\
    \nabla\times u^{P_b}
    &=\omega^{P_b}
    \end{align}
\end{corollary}

\begin{proof}
    Observe that 
    \begin{equation}
    \det(P_f)=\det(P_b)=1,
    \end{equation}
    and the result follows immediately from \Cref{VortThmSO}.
\end{proof}

\begin{remark}
    Note that \Cref{SwapPermuteCor,OtherPermuteCor}
    give a parity result for the behaviour of the vorticity with respect to permuations. For permutations with even parity, $\nabla \times u^P=\omega^P$, and for permutations with odd parity, $\nabla\times u^P
    =-\omega^P$.
\end{remark}

\begin{remark}
    Note that \Cref{VortThmSO} and \Cref{SwapPermuteCor,OtherPermuteCor} also apply in $H^1$ with the caveat that the identities then hold almost everywhere rather than pointwise.
\end{remark}

We will now consider the behaviour of the Fourier transform and the inverse Laplacian with respect to permutation symmetry.

\begin{proposition} \label{FourierProp}
    For all $u\in L^2\left(\mathbb{R}^3;
    \mathbb{R}^3\right)$ and for all $Q\in O(3)$,
    \begin{equation}
    \widehat{u^Q}
    =
    \hat{u}^Q.
    \end{equation}
\end{proposition}

    \begin{proof}
    It is straightforward to compute that for all $u\in L^1\cap L^2$
    \begin{align}
    \widehat{u^Q}(\xi)
    &=
    \int_{\mathbb{R}^3}
    Qu(Q^{tr}x) e^{-2\pi i \xi\cdot x}
    \diff x \\
    &=
    Q\int_{\mathbb{R}^3}
    u(y) e^{-2\pi i \xi\cdot Qy}
    \diff y,
    \end{align}
    making the substitution $y=Q^{tr}x$.
    Next observe that
    \begin{align}
    \xi\cdot Qy
    &=
    \sum_{i,j=1}^3 \xi_i Q_{ij} y_j \\
    &=
    y\cdot Q^{tr}\xi,
    \end{align}
    and so
    \begin{align}
    \widehat{u^Q}(\xi)
    &=
    Q\int_{\mathbb{R}^3}
    u(y) e^{-2\pi i (Q^{tr}\xi)\cdot y}
    \diff y \\
    &=
    \hat{u}^P(\xi).
    \end{align}
    This completes the proof for $u\in L^1\cap L^2$, and the result in general follows because $L^1\cap L^2$ is dense in $L^2$.
    \end{proof}

    \begin{corollary}
    For all $u\in L^2\left(\mathbb{R}^3\right)$,
    $u$ is permutation symmetric, if and only if $\hat{u}$ is permutation symmetric.
    \end{corollary}

    \begin{proof}
    This follows immediately from \Cref{FourierProp}.
    \end{proof}

\begin{lemma} \label{LaplaceLemma}
    For all $f\in L^2$ and for all 
    $Q\in O(3)$,
    \begin{equation}
    (-\Delta)^{-1}\left(u^Q\right)
    =
    \left((-\Delta)^{-1}u\right)^Q.
    \end{equation}
\end{lemma}

\begin{proof}
    We may compute that 
    \begin{align}
    (-\Delta)^{-1}\left(u^Q\right)(x)
    &=
    \int_{\mathbb{R}^3}
    \frac{1}{4\pi |x-y|}
    Q u(Q^{tr}y) \diff y \\
    &=
    Q\int_{\mathbb{R}^3}
    \frac{1}{4\pi |x-Qz|}
    u(z) \diff z \\
    &=
    Q\int_{\mathbb{R}^3}
    \frac{1}{4\pi |Q^{tr}x-z|} 
    u(z) \diff z \\
    &=
    Q(-\Delta)^{-1}u\left(Q^{tr}x\right),
    \end{align}
    where we have taken the substitution $z=Q^{tr}y$ and used the fact that
    $|x-Qz|=|Q^{tr}x-z|$.
\end{proof}

Because the velocity is completely determined by the vorticity by the Biot-Savart law, it is reasonable to expect that permutation symmetry should be able to be expressed as a condition on the vorticity.
With the preliminary results now proven, we can give a necessary and sufficient condition on the vorticity, in order that a velocity field be permutation symmetric.
First, we will state the analogous result for general symmetries.

\begin{theorem} \label{VortGeneralSymThm}
    A vector field $u\in H^s_{df}, s>\frac{5}{2}$ is $Q$-symmetric if and only if
    \begin{equation}
    \omega=\det(Q)\omega^Q.
    \end{equation}
\end{theorem}

\begin{proof}
    Suppose $u=u^Q$. Then taking the curl of each side of this equation and applying \Cref{VortThmSO}, we can see that
    \begin{align}
    \omega
    &=
    \nabla \times u^Q \\
    &=
    \det(Q)\omega^Q.
    \end{align}
    This completes the proof of the forward direction.
    Now suppose that
    \begin{equation}
    \omega=\det(Q)\omega^Q.
    \end{equation}
    Using the Biot-Savart law,
    $u=\nabla\times(-\Delta)^{-1}\omega$,
    and applying \Cref{LaplaceLemma} and \Cref{VortThmSO}, we find that
    \begin{align}
    u
    &=
    \nabla \times (-\Delta)^{-1}\omega \\
    &=
    \det(Q) \nabla\times (-\Delta)^{-1}
    \left(\omega^Q\right) \\
    &=
    \det(Q) \nabla \times 
    \left((-\Delta)^{-1}\omega\right)^Q \\
    &=
    \det(Q)^2 \left(\nabla\times
    (-\Delta)^{-1}\omega\right)^Q \\
    &=
    u^Q,
    \end{align}
    where we have used the fact from \Cref{O3prop} that
    $\det(Q)=\pm 1$. This completes the proof.
\end{proof}

\begin{theorem} \label{VortPermSymThm}
    A vector field $u\in H^s_{df}, s>\frac{5}{2}$ is permutation symmetric if and only if
    \begin{align}
    \omega
    &=
    \omega^{P_f} \\
    &=
    \omega^{P_b} \\
    &=-\omega^{P_{12}} \\
    &=-\omega^{P_{13}} \\
    &=-\omega^{P_{23}}.
    \end{align}
\end{theorem}

\begin{proof}
    This result follows as a fairly immediate corollary of \Cref{VortGeneralSymThm}.
    Recalling that
    \begin{equation}
    \det(P_{12})=\det(P_{13})=\det(P_{23})=-1,
    \end{equation}
    we can see that:
    $u=u^{P_{12}}$ 
    if and only if
    $\omega=-\omega^{P_{12}}$;
    $u=u^{P_{13}}$ 
    if and only if
    $\omega=-\omega^{P_{13}}$;
    $u=u^{P_{23}}$ 
    if and only if
    $\omega=-\omega^{P_{23}}$.
    Likewise recalling that
    \begin{equation}
    \det(P_f)=\det(P_b)=1,
    \end{equation}
    we can see that:
    $u=u^{P_f}$
    if and only if
    $\omega=\omega^{P_f}$;
    $u=u^{P_b}$
    if and only if
    $\omega=\omega^{P_b}$.
    This completes the proof.
\end{proof}

\begin{corollary}
    Let $v\in \mathbb{R}^3, |v|=1$ be a unit vector.
    A vector field $u\in H^s_{df}, 
    s>\frac{5}{2}$,
    is $v$-mirror symmetric if and only if
    \begin{equation}
    \omega=-\omega^{M_v},
    \end{equation}
    where 
    \begin{equation}
    M_v=I_3- 2v\otimes v.
    \end{equation}
\end{corollary}

\begin{proof}
    Observe that $\det(M_v)=-1$, and the result follows immediately from \Cref{VortGeneralSymThm}.
\end{proof}

Permutation symmetry also gives us certain isometries for the Sobolev and Lebesgue norms.

\begin{proposition} \label{VortNormProp}
    For all permutation symmetric, divergence free vector fields $u\in H^{s+1}_{df}, s>\frac{3}{2}$.
    We have the following identities for the Sobolev and $L^p$ norms:
    \begin{equation}
    \|\omega_1\|_{H^s\cap\dot{H}^{-1}}^2
    =
    \|\omega_2\|_{H^s\cap\dot{H}^{-1}}^2
    =
    \|\omega_3\|_{H^s\cap\dot{H}^{-1}}^2
    =
    \frac{1}{3}\|u\|_{H^{s+1}}^2.
    \end{equation}
    For all $2\leq p\leq +\infty$,
    \begin{equation}
    \|\omega_1\|_{L^p}
    =
    \|\omega_2\|_{L^p}
    =
    \|\omega_3\|_{L^p},
    \end{equation}
    and so
    \begin{equation}
    \|\omega\|_{L^p}
    \leq 
    3 \|\omega_1\|_{L^p}.
    \end{equation}
    Note that this implies that if $u,\Tilde{u}\in H^{s+1}_{df}$ are permutation symmetric and 
    $\omega_1=\Tilde{\omega}_1,$
    then $u=\Tilde{u}$.
\end{proposition}

\begin{proof}
    Because $u$ is permutation symmetric, we can see that
    \begin{align}
    \hat{\omega}_2(\xi)
    &=
    -\hat{\omega}_1(P_{12}\xi) \\
    \hat{\omega}_3(\xi)
    &=
    -\hat{\omega}_1(P_{13}\xi).
    \end{align}
    This immediately implies that
    \begin{equation}
    \|\omega_1\|_{H^s\cap\dot{H}^{-1}}^2
    =
    \|\omega_2\|_{H^s\cap\dot{H}^{-1}}^2
    =
    \|\omega_3\|_{H^s\cap\dot{H}^{-1}}^2
    =
    \frac{1}{3}
    \|\omega\|_{H^s\cap\dot{H}^{-1}}^2.
    \end{equation}
    It remains to show that
    \begin{equation}
    \|u\|_{H^{s+1}}^2
    =
    \|\omega\|_{H^s\cap\dot{H}^{-1}}^2.
    \end{equation}
    Using the fact that 
    \begin{equation}
    |\hat{\omega}(\xi)|=2\pi|\xi| |\hat{u}(\xi)|,
    \end{equation}
    we find that
    \begin{align}
    \int_{\mathbb{R}^3}
    \frac{\left(1+4\pi^2|\xi|^2\right)^{s+1}}{4\pi^2|\xi|^2}
    \left|\hat{\omega}(\xi)\right|^2
    \diff\xi
    =
    \int_{\mathbb{R}^3}
    \left(1+4\pi^2|\xi|^2\right)^{s+1}
    \left|\hat{u}(\xi)\right|^2,
    \end{align}
    and the isometry between the velocity and the vorticity holds.

    Now we move to the Lebesgue space bounds.
    Observing that
    \begin{align}
    \omega_2(x)
    &=
    -\omega_1(P_{12}x) \\
    \omega_3(x)
    &=
    -\omega_1(P_{13}x),
    \end{align}
    we can see immediately that
    \begin{equation}
    \|\omega_1\|_{L^p}
    =
    \|\omega_2\|_{L^p}
    =
    \|\omega_3\|_{L^p}.
    \end{equation}
    Applying the triangle inequality we can see that
    \begin{align}
    \|\omega\|_{L^p}
    &\leq 
    \|\omega_1\|_{L^p}
    +\|\omega_2\|_{L^p}
    +\|\omega_3\|_{L^p} \\
    &=
    3\|\omega_1\|_{L^p},
    \end{align}
    and this completes the proof.
\end{proof}

\begin{remark}
    We will conclude this section by observing that while the permutation group has six elements, it can be generated by two permutations. Therefore, it is sufficient to check that a vector field is symmetric with respect to $P_{12}$ and $P_{13}$ to conclude that it is permutation symmetric. This is also true for the vorticity, with the negative sign imposed by parity. An analogous result also holds for $P_{12}$ and $P_f$
\end{remark}

\begin{theorem} \label{PermuteGenerateThm}
    A vector field $u\in H^s\left(\mathbb{R}^3;\mathbb{R}^3\right), s>\frac{3}{2}$
    is permutation symmetric if and only if
    \begin{align}
    u^{P_{12}}&=u \\
   u^{P_{13}}&= u.
    \end{align}
\end{theorem}

\begin{proof}
    It is obvious by definition that if $u$ is permutation symmetric, then the identity holds. What is necessary is to show that symmetry with respect to $P_{12}$ and $P_{13}$ is sufficient to guarantee permutation symmetry. This is true because these two permutations generate $\mathcal{P}_3$.
    First observe that
    \begin{equation}
    P_b=P_{12}P_{13}.
    \end{equation}
    Applying \Cref{OrthoComposeProp}, we can see that
    \begin{equation}
    u^{P_b}
    =
    u^{P_{12}P_{13}} 
    =
    \left(u^{P_{13}}\right)^{P_{12}}
    =
    u^{P_{12}}
    =
    u.
    \end{equation}
    Next observe that
    \begin{equation}
    P_f=P_b^2,
    \end{equation}
    and so
    \begin{equation}
    u^{P_f}
    =
    u^{P_b^2}
    =
    \left(u^{P_b}\right)^{P_b}
    =
    u^{P_b}
    =
    u.
    \end{equation}
    Finally observe that 
    \begin{equation}
    P_{23}=P_{13}P_b,
    \end{equation}
    and so
    \begin{equation}
    u^{P_{23}}
    =
    u^{P_{13}P_b}
    =
    \left(u^{P_{b}}\right)^{P_{13}}
    =
    u^{P_{13}}
    =
    u.
    \end{equation}
    Therefore, $u$ is permutation symmetric.
\end{proof}

\begin{corollary} \label{PermuteGenerateCor}
    A vector field $u\in H^s_{df}, s>\frac{5}{2}$ is permutation symmetric if and only if
    \begin{align}
    \omega^{P_{12}} &= -\omega \\
    \omega^{P_{13}} &= -\omega.
    \end{align}
\end{corollary}

\begin{proof}
    The forward direction was already proven in \Cref{VortPermSymThm}, so we only need to show the reverse direction. We will do this by proving that $u$ is symmetric with respect to $P_{12}$ ad $P_{13}$ and applying \Cref{PermuteGenerateThm}.
    We have already shown in the proof of \Cref{VortPermSymThm} that if 
    \begin{equation}
    \omega^{P_{12}} = -\omega,
    \end{equation}
    then
    \begin{equation}
    u^{P_{12}}=u,
    \end{equation}
    and that if 
    \begin{equation}
    \omega^{P_{13}} = -\omega,
    \end{equation}
    then
    \begin{equation}
    u^{P_{13}}=u.
    \end{equation}
    Combined with \Cref{PermuteGenerateThm},
    this completes the proof.
\end{proof}

\begin{corollary} \label{PermuteGenerateCor2}
    A vector field $u\in H^s_{df}, s>\frac{3}{2}$ is permutation symmetric if and only if
    \begin{align}
    u^{P_{12}} &= u \\
    u^{P_f} &= u.
    \end{align}
\end{corollary}

\begin{proof}
    The proof comes down to the fact that $P_{12}$ and $P_f$ also generate $\mathcal{P}_3$.
    Suppose that 
    \begin{align}
    u^{P_{12}} &= u \\
    u^{P_f} &= u.
    \end{align}
    Then we can see that 
    \begin{equation}
    P_{13}=P_f P_{12},
    \end{equation}
    and so
    \begin{equation}
    u^{P_{13}}
    =u^{P_f P_{12}}
    \left(u^{P_{12}}\right)^{P_f}
    =u^{P_f}
    =u.
    \end{equation}
    Applying \Cref{PermuteGenerateThm}, this completes the proof.
\end{proof}

\begin{corollary} \label{PermuteGenerateCor3}
    A vector field $u\in H^s_{df}, s>\frac{5}{2}$ is permutation symmetric if and only if
    \begin{align}
    \omega^{P_{12}} &= -\omega \\
    \omega^{P_f} &= \omega.
    \end{align}
\end{corollary}

\begin{proof}
    Suppose that
    \begin{align}
    \omega^{P_{12}} &= -\omega \\
    \omega^{P_f} &= \omega.
    \end{align}
    We have shown in the proof of \Cref{VortPermSymThm} that this implies
    \begin{align}
    u^{P_{12}} &= u \\
    u^{P_f} &= u.
    \end{align}
    Applying \Cref{PermuteGenerateCor2},
    this completes the proof.
\end{proof}

\subsection{The permutation symmetric Biot-Savart law}

When the velocity is permutation symmetric, a single vorticity component determines the entire vorticity by taking permutations, and consequently determines the entire velocity.
In this subsection, we will derive a Biot-Savart law in terms of one vorticity component $\omega_1$.

\begin{theorem} \label{BiotSavart}
    Suppose $u\in H^s_{df}, s>\frac{5}{2}$ is permutation symmetric. Then $u$ can be determined entirely by one component of the vorticity, $\omega_1$, and has the Biot-Savart law
    \begin{equation}
    u(x)=\int_{\mathbb{R}^3}
    G(x,y) \omega_1(y) \diff y,
    \end{equation}
    where 
    \begin{equation}
    G(x,y)=
    \frac{1}{|x-y|^3}\left(\begin{array}{c}
         0  \\ -x_3+y_3 \\ x_2-y_2 
    \end{array}\right)
    +
    \frac{1}{|x-P_{12}(y)|^3}
    \left(\begin{array}{c}
         -x_3+y_3  \\ 0 \\ x_1-y_2 
    \end{array}\right)
    +
    \frac{1}{|x-P_{13}(y)|^3}
    \left(\begin{array}{c}
         x_2-y_2  \\ -x_1+y_3 \\ 0 
    \end{array}\right).
    \end{equation}
\end{theorem}

\begin{proof}
    Recalling that the classic Biot-Savart law is given by $\omega=\nabla\times(-\Delta)^{-1}\omega$, we compute that
    \begin{align}
    u(x)
    &=
    \nabla \times\int_{\mathbb{R}^3}
     \left(\frac{1}{4\pi|x-y|}\right)
     \omega(y)\diff y \\
     &=
     -\int_{\mathbb{R}^3}
     \frac{x-y}{4\pi |x-y|^3}\times \omega(y) 
     \diff y \\
     &=
     \int_{\mathbb{R}^3}
     \frac{1}{4\pi |x-y|^3}\left(
     \left(\begin{array}{c}
       0  \\ -x_3+y_3 \\ x_2-y_2 
     \end{array}\right) \omega_1(y)
     +\left(\begin{array}{c}
        x_3-y_3 \\ 0 \\ -x_1+y_1
     \end{array}\right) \omega_2(y)
     +\left(\begin{array}{c}
        -x_2+y_2 \\ x_1-y_1 \\ 0
     \end{array}\right) \omega_3(y)
     \right)\diff y.
    \end{align}
    Applying \Cref{VortPermSymThm} we can see that
    \begin{equation}
    \omega_2(y)=-\omega_1(P_{12}(y))
    \end{equation}
    and letting $z=P_{12}(y)$, we find that
    \begin{align}
    \frac{1}{4\pi}\int_{\mathbb{R}^3}
     \frac{1}{|x-y|^3}
     +\left(\begin{array}{c}
        x_3-y_3 \\ 0 \\ -x_1+y_1
     \end{array}\right) \omega_2(y) \diff z
     &=
     -\frac{1}{4\pi}\int_{\mathbb{R}^3}
     \frac{1}{|x-y|^3}
     +\left(\begin{array}{c}
        x_3-y_3 \\ 0 \\ -x_1+y_1
     \end{array}\right) 
     \omega_1(P_{12}(y)) \diff y \\
     &=
     \frac{1}{4\pi}\int_{\mathbb{R}^3}
     \frac{1}{|x-P_{12}(z)|^3}
     +\left(\begin{array}{c}
        -x_3+z_3 \\ 0 \\ x_1-z_2
     \end{array}\right) 
     \omega_1(z) \diff z.
    \end{align}
    We likewise observe that
    \begin{equation}
    \omega_3(y)=-\omega_1(P_{13}(y)),
    \end{equation}
    and letting $z=P_{13}(y)$, we can compute that
    \begin{align}
    \frac{1}{4\pi}\int_{\mathbb{R}^3}
     \frac{1}{|x-y|^3}
     \left(\begin{array}{c}
        -x_2+y_2 \\ x_1-y_1 \\ 0
     \end{array}\right) \omega_3(y)
     \diff y 
     &=
     -\frac{1}{4\pi}\int_{\mathbb{R}^3}
     \frac{1}{|x-y|^3}
     \left(\begin{array}{c}
        -x_2+y_2 \\ x_1-y_1 \\ 0
     \end{array}\right) \omega_1(P_{13}(y))
     \diff y \\
     &=
     \frac{1}{4\pi}\int_{\mathbb{R}^3}
     \frac{1}{|x-P_{13}(z)|^3}
     \left(\begin{array}{c}
        x_2-z_2 \\ -x_1+z_3 \\ 0
     \end{array}\right) \omega_1(z)
     \diff z.
    \end{align}
    Changing the variable of integration back from $z$ to $y$ for consistency of these latter two integrals with the first integral, this completes the proof.
\end{proof}

\subsection{The evolution equation for one vorticity component}

We have seen in the previous subsection that, for permutation-symmetric vector fields, one component of the vorticity completes determines both the velocity and the vorticity.
We will now show that this implies the vorticity evolution equation can be reduced to an evolution equation for a single component: 
\begin{equation} \label{OneCompCorEqn}
    \partial_t \omega_1
    +(u\cdot\nabla)\omega_1
    -(\omega\cdot\nabla) u_1
    =0.
\end{equation}
This is an advantage, because a scalar vorticity can satisfy sign conditions, which the vorticity vector cannot. This means a scalar vorticity is more amenable to blowup arguments involving a quantity bounded below. In this subsection, we will prove \Cref{OneVortCompThmIntro}, which will be broken up into multiple pieces. We begin by stating some classical results for the Euler equation, the local existence theory, the equivalence of the Euler equation and the Euler vorticity equation, and the the invariance of the Euler equation with respect to $O(3)$. See \cites{KatoPonce,BKM,MajdaBertozzi} for details.

\begin{theorem} \label{EulerExistThm}
    For all $u^0\in H^s_{df}, s>\frac{5}{2}$, there exists a unique strong solution of the Euler equation
    $u\in C\left([0,T_{max});H^s_{df}\right) \cap
    C^1\left([0,T_{max});H^{s-1}_{df}\right)$,
    where 
    $T_{max}\geq \frac{C_s}{\left\|u^0\right\|_{H^s}}$.
    This solution satisfies the energy equality: for all $0<t<T_{max}$,
    \begin{equation}
    \|u(\cdot,t)\|_{L^2}^2=\left\|u^0\right\|_{L^2}^2.
    \end{equation}
    Furthermore, if $T_{max}<+\infty$, then
    \begin{equation}
    \int_0^{T_{max}} 
    \|\omega(\cdot,t)\|_{L^\infty}\diff t
    =+\infty.
    \end{equation}
\end{theorem}

\begin{theorem} \label{EulerVortEquivThm}
    Suppose
    $u\in C\left([0,T_{max});H^s_{df}\right) \cap
    C^1\left([0,T_{max});H^{s-1}_{df}\right),
    s>\frac{5}{2}$.
    Then $u$ is a solution of the Euler equation if and only if $\omega=\nabla\times u$ is a solution of the Euler vorticity equation
    \begin{equation}
    \partial_t\omega
    +(u\cdot\nabla)\omega
    -(\omega\cdot\nabla)u
    =0.
    \end{equation}
\end{theorem}

\begin{proposition} \label{O3InvariantProp}
    Suppose $u\in C\left([0,T_{max});H^s_{df}\right) \cap
    C^1\left([0,T_{max});
    H^{s-1}_{df}\right)$
    is a solution of the Euler equation and $Q\in O(3)$.
    Then $u^Q\in C\left([0,T_{max});H^s_{df}\right) \cap
    C^1\left([0,T_{max});
    H^{s-1}_{df}\right)$
    is also a solution of the Euler equation.
\end{proposition}

The subspace of permutation symmetric velocity fields is preserved by the dynamics of the Euler equation. This result was proven in \cite{MillerRestricted}, but the arguments are entirely classical, following immediately from \Cref{O3InvariantProp} and uniqueness.

\begin{proposition} \label{PermuteEulerProp}
    Suppose $u^0\in H^s_{df}, s>\frac{5}{2}$ is permutation symmetric. Then the solution of the Euler equation with this initial data $u\in C\left([0,T_{max});H^s_{df}\right) \cap
    C^1\left([0,T_{max});H^{s-1}_{df}\right)$
    is also permutation symmetric for all $0<t<T_{max}$.
    Furthermore, suppose $u^0$ is odd, permutation symmetric, and $\sigma$-mirror symmetric. Then $u(\cdot,t)$ is also odd, permutation symmetric, and $\sigma$-mirror symmetric
    for all $0<t<T_{max}$.
\end{proposition}

This result will allow us to find solutions of the single-component, permutation-symmetric, Euler vorticity equation as an immediate corollary. We will then use a stability estimate to prove uniqueness. First we must establish a relationship between the constraint space $\dot{H}^s_*$ and functions that are components of the vorticities of permutation symmetric vector fields.

\begin{theorem} \label{ConstraintThm}
    Suppose $u\in H^{s+1}_{df}$ is permutation symmetric and $\omega=\nabla \times u$.
    Then $\omega_1\in H^s_*\cap \dot{H}^{-1}$.
    Suppose $f\in H^s_*\cap \dot{H}^{-1}$.
    Then there exists a unique permutation symmetric vector field $\Tilde{u}\in H^{s+1}_{df}$,
    such that $f=\Tilde{\omega}_1$, where 
    $\Tilde{\omega}=
    \nabla \times \Tilde{u}$.
\end{theorem}

\begin{proof}
    Recall that the condition for $f$ to be in the constraint space is
    \begin{equation}
    f(x)=-f(P_{23}x),
    \end{equation}
    and that in Fourier space
    \begin{equation}
    \xi_1 \hat{f}(\xi) 
    -\xi_2 \hat{f}(P_{12}\xi)
    -\xi_3 \hat{f}(P_{13}\xi)
    =0.
    \end{equation}
    Suppose $u\in H^{s+1}_{df}$ is permutation symmetric and $\omega=\nabla \times u$.
    Then applying \Cref{VortPermSymThm}, we can see that
    \begin{align}
    \omega&=- \omega^{P_{12}}\\
            &=- \omega^{P_{13}}.
    \end{align}
    Applying \Cref{FourierProp}, we can then see that
    \begin{align}
\hat{\omega}&=-\hat{\omega}^{P_{12}} \\
        &=- \hat{\omega}^{P_{13}}.
    \end{align}
    This implies that
    \begin{align}
    \hat{\omega}_2(\xi)
    &= 
    -\hat{\omega}_1(P_{12}\xi) \\
    \hat{\omega}_3(\xi)
    &= 
    -\hat{\omega}_1(P_{13}\xi).
    \end{align}
    The divergence free constraint $\nabla\cdot\omega=0$ can be expressed in Fourier space as
    \begin{equation}
    \xi\cdot\hat{\omega}(\xi)=0,
    \end{equation}
    and so
    \begin{equation}
    \xi_1\hat{\omega}_1(\xi)
    -\xi_2\hat{\omega}_1(P_{12}\xi)
    -\xi_3\hat{\omega}_1(P_{13}\xi)=0.
    \end{equation}
    Finally, observe that $\omega=-\omega^{P_{23}}$,
    which implies
    \begin{equation}
    \omega_1(x)=-\omega_1(P_{23}x).
    \end{equation}
    Therefore, we can conclude that $\omega_1\in H^s_*\cap \dot{H}^{-1}$.

    Now suppose $f\in H^s_*\cap \dot{H}^{-1}$.
    We will define $\omega$ (distinct from the vorticity above) by 
    \begin{equation}
    \omega(x)
    =
    \left(\begin{array}{c}
         f(x) \\
        -f(P_{12}x) \\
        -f(P_{13}x) 
    \end{array}\right).
    \end{equation}
    It then follows that
    \begin{equation}
    \hat{\omega}(\xi)
    =
    \left(\begin{array}{c}
         \hat{f}(\xi) \\
        -\hat{f}(P_{12}\xi) \\
        -\hat{f}(P_{13}\xi)  
    \end{array}\right),
    \end{equation}
    and so we can compute that
    \begin{equation}
    \xi\cdot \hat{\omega}(\xi)
    =
    \xi_1 \hat{f}(\xi)
    -\xi_2 \hat{f}(P_{12}\xi)
    -\xi_3 \hat{f}(P_{13}\xi)
    =0.
    \end{equation}
    Therefore $\nabla\cdot\omega=0$.
    Define $u\in H^{s+1}_{df}$ by
    \begin{equation}
    u=\nabla\times (-\Delta)^{-1}\omega.
    \end{equation}
    The Biot-Savart law implies that
    \begin{equation}
    \omega=\nabla\times u,
    \end{equation}
    so it only remains to show that $u$ is permutation symmetric. 
    
    We will do this using \Cref{PermuteGenerateCor}.
    Observe that
    \begin{equation}
    P_b= 
    P_{12}P_{13}=P_{23}P_{12}=P_{13}P_{23},
    \end{equation}
    and that
    \begin{equation}
    P_{f}=
    P_{13}P_{12}=P_{12}P_{23}=P_{23}P_{13}.
    \end{equation}
    Observe that
    \begin{align}
    \omega^{P_{12}}(x)
    &=
    \left(\begin{array}{c}
         -f(P_{12}^2 x)  \\
         f(P_{12}x) \\
         -f(P_{13} P_{12}x) 
    \end{array}\right) \\
    &=
    \left(\begin{array}{c}
         -f(x)  \\
         f(P_{12}x) \\
         -f(P_{23} P_{13}x) 
    \end{array}\right) \\
    &=
    \left(\begin{array}{c}
         -f(x)  \\
         f(P_{12}x) \\
         f( P_{13}x) 
    \end{array}\right) \\
    &=-\omega(x).
    \end{align}
    Note that in the last step we used the fact that $-f(P_{23}y)=f(y)$ for all $y\in\mathbb{R}^3$, and let $y=P_{13}x$.
    Proceeding similarly for $P_{13}$, we find that
    \begin{align}
    \omega^{P_{13}}(x)
    &=
    \left(\begin{array}{c}
    -f(P_{13}^2 x) \\
    -f(P_{12}P_{13}x) \\
    f(P_{13}x)
    \end{array}\right) \\
    &=
    \left(\begin{array}{c}
    -f(x) \\
    -f(P_{23}P_{12}x) \\
    f(P_{13}x)
    \end{array}\right) \\
    &=
    \left(\begin{array}{c}
    -f(x) \\
    f(P_{12}x) \\
    f(P_{13}x)
    \end{array}\right) \\
    &=
    -\omega(x).
    \end{align}
    We have shown that
    \begin{equation}
    \omega^{P_{12}}=\omega^{P_{13}}=-\omega,
    \end{equation}
    and so applying \Cref{PermuteGenerateCor},
    we can conclude that $u$ is permutation symmetric.
    Finally, note that \Cref{VortNormProp} implies the uniqueness of $u$, and this completes the proof.
\end{proof}

\begin{corollary} \label{ExistOneCompCor}
    Suppose $\omega_1^0\in H^s_*\cap \dot{H}^{-1}, s>\frac{3}{2}$. Then there exists a solution to the single-component, permutation-symmetric  Euler vorticity equation 
    $\omega_1\in C\left([0,T_{max});H^s_*\cap \dot{H}^{-1}\right) \cap
    C^1\left([0,T_{max});H^{s-1}_*\right)$,
    where
    \begin{align}
    \omega_2(x)
    &=
    -\omega_1(x_2,x_1,x_3) \\
    \omega_3(x)
    &=
    -\omega_1(x_3,x_2,x_1).
    \end{align}
    \begin{equation}
    u(x)=\int_{\mathbb{R}^3}
    G(x,y) \omega_1(y) \diff y,
    \end{equation}
    where
    \begin{equation}
    G(x,y)=
    \frac{1}{|x-y|^3}\left(\begin{array}{c}
         0  \\ -x_3+y_3 \\ x_2-y_2 
    \end{array}\right)
    +
    \frac{1}{|x-P_{12}(y)|^3}
    \left(\begin{array}{c}
         -x_3+y_3  \\ 0 \\ x_1-y_2 
    \end{array}\right)
    +
    \frac{1}{|x-P_{13}(y)|^3}
    \left(\begin{array}{c}
         x_2-y_2  \\ -x_1+y_3 \\ 0 
    \end{array}\right).
    \end{equation}
    Furthermore, $u\in C\left([0,T_{max});
    H^{s+1}_{df}\right)
    \cap
    C^1\left([0,T_{max});
    H^s_{df}\right)$ is a permutation symmetric solution of the Euler equation.
    We also have a lower bound on the time of existence
    \begin{equation}
    T_{max}
    \geq 
    \frac{C_s}{\left\|\omega_1^0
    \right\|_{H^s\cap\dot{H}^{-1}}},
    \end{equation}
    where $C_s>0$ is an absolute constant independent of $\omega_1^0$
\end{corollary}

\begin{proof}
    Because $\omega_1^0\in H^s_*\cap \dot{H}^{-1}$, we can conclude from \Cref{ConstraintThm} that there exists a permutation-symmetric, divergence free vector field
    $u^0\in H^{s+1}_{df}$, such that 
    \begin{equation}
    \omega_1^0=\left(\nabla\times u^0\right)_1.
    \end{equation}
    From \Cref{EulerExistThm} and \Cref{PermuteEulerProp},
    we can see that there exists a permutation-symmetric solution of the Euler equation 
    $u\in C\left([0,T_{max});
    H^{s+1}_{df}\right)
    \cap
    C^1\left([0,T_{max});
    H^s_{df}\right)$,
    and that $\omega=\nabla\times u$
    satisfies the vorticity equation
    \begin{equation}
    \partial_t\omega
    +(u\cdot\nabla)\omega
    -(\omega\cdot\nabla)u=0.
    \end{equation}
    Taking the first component of this equation,
    we can see that
    \begin{equation}
    \partial_t\omega_1
    +(u\cdot\nabla)\omega_1
    -(\omega\cdot\nabla)u_1=0.
    \end{equation}
    Because $u$ is permutation symmetric, we can see that
    $\omega_1\in C\left([0,T_{max});H^s_*\cap \dot{H}^{-1}\right) \cap
    C^1\left([0,T_{max});H^{s-1}_*\right)$.
    Applying the lower bound on the time of existence from \Cref{EulerExistThm} and the equivalence of Sobolev norms from \Cref{VortNormProp}, we find that
    \begin{align}
    T_{max}
    &\geq 
    \frac{C'_{s+1}}{\left\|u^0\right\|_{H^{s+1}}} \\
    &= 
    \frac{C'_{s+1}}{\sqrt{3}\left\|\omega_1^0
    \right\|_{H^{s}\cap\dot{H}^{-1}}}.
    \end{align}
    Finally, applying \Cref{VortPermSymThm,BiotSavart}, we can compute $\omega_2,\omega_3$ and $u$ in terms of $\omega_1$.
\end{proof}

\begin{remark}
    We have shown that there exists a solution to the single-component, permutation-symmetric Euler vorticity equation such that the associated velocity $u$ is a solution to the Euler equation, but we have not yet shown that solutions of the single-component, permutation-symmetric Euler vorticity equation are unique. Solutions of the Euler equation are unique, but we have not proven that the equivalence of solutions of these equations, because there are also evolution equations for $\omega_2$ and $\omega_3$ that must be satisfied. We will prove that solutions of the evolution equation for $\omega_1$ are unique by means of a stability estimate, and uniqueness will guarantee the equivalence of solutions of the single-component, permutation-symmetric Euler vorticity equation and the Euler equation when considering permutation symmetric flows.
\end{remark}

\begin{theorem} \label{StabilityThm}
    Suppose $\omega_1,\Tilde{\omega}_1\in C\left([0,T];H^s_*\cap \dot{H}^{-1}\right) \cap
 C^1\left([0,T];H^{s-1}_*\right), s>\frac{3}{2}$,
are solutions to the single-component, permutation-symmetric  Euler vorticity equation. Then for all $0\leq t\leq T$,
\begin{equation}
    \|(\omega_1-\Tilde{\omega}_1)
    (\cdot,t)\|_{L^2}^2
    \leq 
    \left\|\omega_1^0-\Tilde{\omega}_1^0
    \right\|_{L^2}^2
    \exp\left(C_s \int_0^t
    \|(\omega_1+\Tilde{\omega}_1)
    (\cdot,\tau)\|_{H^s}\diff\tau\right),
\end{equation}
where $C_s>0$ is an absolute constant independent of $\omega_1^0, \Tilde{\omega}_1^0$.
Note in particular this implies that solutions of the single-component, permutation symmetric Euler vorticity equation are unique.
\end{theorem}

\begin{proof}
    We begin by computing that
    \begin{equation}
    \frac{\diff}{\diff t}
    \|(\omega_1-\Tilde{\omega}_1)
    (\cdot,t)\|_{L^2}^2
    =
    2\left<\omega_1-\Tilde{\omega}_1,
    -(u\cdot\nabla)\omega_1
    +(\omega\cdot\nabla)u_1
    +(\Tilde{u}\cdot\nabla)\Tilde{\omega}_1
    -(\Tilde{\omega}_1\cdot\nabla)\Tilde{u}_1
    \right>.
    \end{equation}
    Rearranging terms we can see that
    \begin{multline}
    \frac{\diff}{\diff t}
    \|(\omega_1-\Tilde{\omega}_1)
    (\cdot,t)\|_{L^2}^2
    =
    -\left<\omega_1-\Tilde{\omega}_1,
    (u+\Tilde{u})\cdot\nabla (\omega_1-\Tilde{\omega}_1)\right>
    -\left<\omega_1-\Tilde{\omega}_1,
    (u-\Tilde{u})\cdot\nabla 
    (\omega_1+\Tilde{\omega}_1)\right> \\
    +\left<\omega_1-\Tilde{\omega}_1,
    (\omega+\Tilde{\omega})\cdot\nabla 
    (u_1-\Tilde{u}_1)\right>
    +\left<\omega_1-\Tilde{\omega}_1,
    (\omega-\Tilde{\omega})\cdot\nabla 
    (u_1+\Tilde{u}_1)
    \right>.
    \end{multline}
    Integrating by parts, we can see that
    \begin{equation}
    -\left<\omega_1-\Tilde{\omega}_1,
    (u+\Tilde{u})\cdot\nabla (\omega_1-\Tilde{\omega}_1)\right>
    =
    \left<\omega_1-\Tilde{\omega}_1,
    (u+\Tilde{u})\cdot\nabla (\omega_1-\Tilde{\omega}_1)\right>
    =0.
    \end{equation}
    Applying H\"older's inequality, the Sobolev inequality, and \Cref{VortNormProp}, 
    we find that
    \begin{align}
    -\left<\omega_1-\Tilde{\omega}_1,
    (u-\Tilde{u})\cdot\nabla 
    (\omega_1+\Tilde{\omega}_1)\right> 
    &\leq 
    \|\omega_1-\Tilde{\omega}_1\|_{L^2}
    \|u-\Tilde{u}\|_{L^6}
    \|\nabla 
    (\omega_1+\Tilde{\omega}_1)\|_{L^3} \\
    &\leq 
    C \|\omega_1-\Tilde{\omega}_1\|_{L^2}^2
    \|\omega_1+\Tilde{\omega}_1\|_{H^\frac{3}{2}},
    \end{align}
    that
    \begin{align}
    +\left<\omega_1-\Tilde{\omega}_1,
    (\omega+\Tilde{\omega})\cdot\nabla 
    (u_1-\Tilde{u}_1)\right>
    &\leq 
    \|\omega_1-\Tilde{\omega}_1\|_{L^2} \|\omega+\Tilde{\omega}\|_{L^\infty}
    \|\nabla(u_1-\Tilde{u}_1)\|_{L^2} \\
    &\leq 
    C \|\omega_1-\Tilde{\omega}_1\|_{L^2}^2
    \|\omega_1+\Tilde{\omega}_1\|_{H^s},
    \end{align}
    and that
    \begin{align}
    \left<\omega_1-\Tilde{\omega}_1,
    (\omega-\Tilde{\omega})\cdot\nabla 
    (u_1+\Tilde{u}_1)\right>
    &\leq 
    \|\omega_1-\Tilde{\omega}_1\|_{L^2}
    \|\omega-\Tilde{\omega}\|_{L^2} 
    \|\nabla (u_1+\Tilde{u}_1)\|_{L^\infty} \\
    &\leq 
    C \|\omega_1-\Tilde{\omega}_1\|_{L^2}^2
    \|\omega_1+\Tilde{\omega}_1\|_{H^s}.
    \end{align}
    Putting these together we find that
    for all $0<t<T_{max}$
    \begin{equation}
    \frac{\diff}{\diff t}
    \|(\omega_1-\Tilde{\omega}_1)
    (\cdot,t)\|_{L^2}^2
    \leq 
    C \|\omega_1+\Tilde{\omega}_1\|_{H^s} 
    \|\omega_1-\Tilde{\omega}_1\|_{L^2}^2.
    \end{equation}
    Applying Gr\"onwall's inequality, this completes the proof.
\end{proof}

\begin{theorem} $\label{CompReduceThm}$
    Suppose $\omega_1\in C\left([0,T];H^s_*\cap \dot{H}^{-1}\right) \cap
 C^1\left([0,T];H^{s-1}_*\right), s>\frac{3}{2}$
 is the unique solution of the single-component, permutation-symmetric Euler vorticity equation. Then $u\in C\left([0,T_{max};H^{s+1}_{df}\right) 
    \cap C^1\left([0,T_{max};
    H^{s}_{df}\right)$ is a permutation symmetric solution of the Euler equation,
    where $u$ is defined in terms $\omega_1$
    as in as in \Cref{BiotSavart} by
    \begin{equation}
    u(x,t)
    =
    \int_{\mathbb{R}^3}
    G(x,y)\omega_1(y,t)\diff y.
    \end{equation}
\end{theorem}

\begin{proof}
We showed in \Cref{ExistOneCompCor} that there exists a solution of the the Euler equation whose vorticity satisfies \eqref{OneCompCorEqn}.
\Cref{StabilityThm} implies that solutions of the single-component, permutation-symmetric Euler vorticity equation are unique, and the first component of the vorticity of this solution of the Euler equation $\omega_1$ must give the only solution of \eqref{OneCompCorEqn}. This completes the proof.
\end{proof}

\begin{corollary}
    Suppose $u\in C\left([0,T_{max};H^s_{df}\right) 
    \cap C^1\left([0,T_{max};
    H^{s-1}_{df}\right), s>\frac{5}{2}$ is permutation symmetric.
    Then $u$ is a solution of the Euler equation if and only if 
    \begin{equation}
    \partial_t\omega_1
    +(u\cdot\nabla)\omega_1
    -(\omega\cdot\nabla)u_1=0.
    \end{equation}
\end{corollary}

\begin{proof}
    This follows immediately from \Cref{ConstraintThm,CompReduceThm}.
\end{proof}

We have completed the proof of every part of \Cref{OneVortCompThmIntro} except the energy equality and the Beale-Kato-Majda criterion; we will now prove variants of both of these for the single-component, permutation-symmetric Euler vorticity equation.

\begin{corollary}
Suppose $\omega_1\in C\left([0,T_{max});H^s_*\cap \dot{H}^{-1}\right) \cap
 C^1\left([0,T_{max});H^{s-1}_*\right)$,
is the unique solution to the single-component, permutation-symmetric Euler vorticity equation.
Then for all $0<t<T_{max}$, 
    \begin{equation}
    \|\omega_1(\cdot,t)
    \|_{\dot{H}^{-1}}^2
    =
    \left\|\omega_1^0
    \right\|_{\dot{H}^{-1}}^2.
    \end{equation}
\end{corollary}

\begin{proof}
    \Cref{CompReduceThm} implies that $u$ is a solution of the Euler equation.
    We know from the energy equality that 
    for all $0\leq t<T_{max}$,
    \begin{equation}
    \|u(\cdot,t)\|_{L^2}^2=
    \left\|u^0\right\|_{L^2}^2.
    \end{equation}
    Following the methods of \Cref{VortNormProp}, we can see that because $u(\cdot,t)$ is a permutation-symmetric, divergence free vector field,
    \begin{equation}
    \|\omega_1\|_{\dot{H}^{-1}}^2
    =
    \|\omega_2\|_{\dot{H}^{-1}}^2
    =
    \|\omega_3\|_{\dot{H}^{-1}}^2
    =
    \frac{1}{3}\|u\|_{L^2}^2,
    \end{equation}
    and the result follows.
\end{proof}

\begin{corollary}
Suppose $\omega_1\in C\left([0,T_{max});H^s_*\cap \dot{H}^{-1}\right) \cap
    u\in C^1\left([0,T_{max});H^{s-1}_*\right)$,
is the unique solution to the single-component, permutation-symmetric Euler vorticity equation.
Then if $T_{max}<+\infty$,
    \begin{equation}
    \int_0^{T_{max}}
    \|\omega_1(\cdot,t)\|_{L^\infty}
    \diff t
    =+\infty.
    \end{equation}
\end{corollary}

\begin{proof}
    This follows immediately from the standard Beale-Kato-Majda criterion, \Cref{VortNormProp}, and \Cref{CompReduceThm}.
\end{proof}

\section{The geometry of permutation symmetric flows}

In this section, we will consider the geometric structure of permutation symmetric flows. We will be particularly interested in the behaviour of the strain matrix at the origin, and using the Biot-Savart law to identify vorticity structures likely to produce singular behaviour.

\subsection{Vorticity and the strain at the origin}

We begin with the computation of the velocity gradient at the origin, recalling a structure result from \cite{MillerRestricted}.

\begin{proposition} \label{SigmaAxisGradientProp}
    Suppose $u\in H^s_{df}, s>\frac{5}{2}$ is permutation symmetric and $x_1=x_2=x_3$.
    Then for all $i\neq j$
    \begin{equation}
    \partial_1 u_2(x)=\partial_i u_j(x),
    \end{equation}
    and for all $1\leq i\leq 3$,
    \begin{equation}
    \partial_i u_i(x)=0.
    \end{equation}
    Note that this also implies that
    \begin{equation}
    \omega(x)=0.
    \end{equation}
\end{proposition}

Not only do we know that the gradient has this structure on the $\sigma$-axis. At the origin we have even more cancellation when we compute the full gradient in terms of vorticity using the Biot-Savart law.

\begin{theorem} \label{GradientOriginThm}
    Suppose $u\in H^s_{df}, s>\frac{5}{2}$ is permutation symmetric.
    Then the gradient at the origin is given by
    \begin{equation}
    \nabla u(0)
    =
    \lambda 
    \left( \begin{array}{ccc}
        0 & -1 & -1  \\
        -1 & 0 & -1  \\
        -1 & -1 & 0
    \end{array}\right),
    \end{equation}
    where
    \begin{align}
    \lambda
    &= \label{LambdaEqn1}
    \frac{3}{8\pi}\int_{\mathbb{R}^3}
    \left(\frac{(\sigma\cdot x)(x_2-x_3)}
    {|x|^5}\right)\omega_1(x) \diff x \\
    &= \label{LambdaEqn2}
    \frac{3}{8\pi}\int_{\mathbb{R}^3}
    \left(\frac{(\sigma\cdot x)(x_3-x_1)}
    {|x|^5}\right)\omega_2(x) \diff x \\
    &=  \label{LambdaEqn3}
    \frac{3}{8\pi}\int_{\mathbb{R}^3}
    \left(\frac{(\sigma\cdot x)(x_1-x_2)}
    {|x|^5}\right)\omega_3(x) \diff x.
    \end{align}
\end{theorem}

\begin{proof}
    Applying \Cref{SigmaAxisGradientProp}, we can see that the result holds with 
    \begin{equation}
    \lambda
    =
    -\partial_2 u_1(0),
    \end{equation}
    so it is simply a question of computing this partial derivative.
    First, recall from the permutation symmetric Biot-Savart law in \Cref{BiotSavart}, that
    \begin{equation}
    u_1(x)
    =
    \frac{1}{4\pi}
    \int_{\mathbb{R}^3}
    \left(\frac{x_2-y_2}{|x-P_{13}(y)|^3}-\frac{x_3-y_3}{|x-P_{12}(y)|^3}\right)
    \omega_1(y) \diff y.
    \end{equation}
    Differentiating, and using the principle value for the singular integral operator, we find that
    \begin{equation}
    \partial_2 u_1(x)
    =
    \frac{1}{4\pi} P.V.
    \int_{\mathbb{R}^3}
    \left(\frac{1}{|x-P_{13}(y)|^3}
    -3 \frac{(x_2-y_2)^2}{|x-P_{13}(y)|^5}
    +3\frac{(x_3-y_3)(x_2-y_1)}
    {|x-P_{12}(y)|^5}\right)
    \omega_1(y) \diff y,
    \end{equation}
    where in this case the principal value of the singular integral is defined by
    \begin{equation}
    P.V. \int_{\mathbb{R}^3}
    H(x,y)\omega_1(y)\diff y
    =
    \lim_{\epsilon\to 0}
    \int_{\substack{|x-P_{12}(y)|>\epsilon
    \\ |x-P_{13}(y)|>\epsilon}}
    H(x,y)\omega_1(y)\diff y,
    \end{equation}
    and $H$ is given by the singular kernel above.
    Plugging in $x=0$ and observing that permutations don't affect the magnitude of a vector, we can see that
    \begin{equation}
    \partial_2 u_1(0)
    = \label{Step1}
    \frac{1}{4\pi} P.V.
    \int_{\mathbb{R}^3}
    \left(
    \frac{|y|^2-3y_2^2+3 y_1y_3}{|y|^5}
    \right)
    \omega_1(y)\diff y,
    \end{equation}
    where the principal value is now 
    simplifies to
    \begin{equation}
    P.V.
    \int_{\mathbb{R}^3}
    \left(
    \frac{|y|^2-3y_2^2+3 y_1y_3}{|y|^5}
    \right)
    \omega_1(y) \diff y
    =
    \lim_{\epsilon\to 0}
    \int_{|y|>\epsilon}
    \left(
    \frac{|y|^2-3y_2^2+3 y_1y_3}{|y|^5}
    \right)
    \omega_1(y)\diff y
    \end{equation}
    Using the fact that $\omega_1(z)=
    -\omega_1(P_{23}(z))$, and letting $y=P_{23}(z)$, we can see that
    \begin{align}
    \partial_2 u_1(0)
    &=
    -\frac{1}{4\pi} P.V.
    \int_{\mathbb{R}^3}
    \left(
    \frac{|z|^2-3z_2^2+3z_1z_3}{|z|^5}
    \right)
    \omega_1(P_{23}(z))\diff z \\
    &= \label{Step2}
    \frac{1}{4\pi} P.V.
    \int_{\mathbb{R}^3}
    \left(
    \frac{-|y|^2+3y_3^2-3 y_1y_2}{|y|^5}
    \right)
    \omega_1(y)\diff y.
    \end{align}
    Taking one half the sum of \cref{Step1,Step2}, we find that
    \begin{align}
    \partial_2 u_1(0)
    &=
    \frac{3}{8\pi} P.V.
    \int_{\mathbb{R}^3}
    \left(\frac{y_3^2-y_2^2+y_1y_3-y_1y_2}
    {|y|^5}\right) \omega_1(y) \diff y \\
    &=
    \frac{3}{8\pi} P.V.
    \int_{\mathbb{R}^3}
    \left(\frac{(y_1+y_2+y_3)(y_3-y_2)}
    {|y|^5}\right) \omega_1(y) \diff y,
    \end{align}
    and so we can see that
    \begin{equation}
    \lambda
    =
    \frac{3}{8\pi} P.V.
    \int_{\mathbb{R}^3}
    \left(\frac{(\sigma\cdot y)(y_2-y_3)}
    {|y|^5}\right)\omega_1(y) \diff y.
    \end{equation}

    We will now remove the principal value from this estimate, as the integral is convergent once geometric considerations are accounted for.
    Recall that $\omega(0)=0$ because of permutation symmetry, and that by Sobolev embedding $\omega\in C^\alpha$,
    where $\alpha=s-\frac{5}{2}$. 
    This implies that for all $x\in\mathbb{R}^3$,
    \begin{equation}
    |\omega(x)|\leq C |x|^\alpha,
    \end{equation}
    where $C$ does not depend on $x$.
    This implies that
    \begin{align}
    \left|\int_{|y|\leq \epsilon}
     \left(\frac{(\sigma\cdot y)(y_2-y_3)}
    {|y|^5}\right)\omega_1(y) \diff y \right|
    &\leq 
    \sqrt{6}\int_{|y|\leq \epsilon} 
    \frac{|\omega_1(y)|}{|y|^3}\\
    &\leq 
    \sqrt{6}C 
    \int_{|y|\leq \epsilon} 
    |y|^{\alpha-3} \diff y \\
    &=
    4\sqrt{6}C\pi 
    \int_0^\epsilon r^{\alpha-1} \diff r \\
    &=
    \frac{4\sqrt{6}C\pi}{\alpha}\epsilon^\alpha,
    \end{align}
    where we have use the fact that $\diff y =4\pi r^2 \diff r$ in the presence of spherical symmetry.
    This clearly implies that
    \begin{equation}
    \lim_{\epsilon\to 0}
    \int_{|y|\leq \epsilon}
     \left(\frac{(\sigma\cdot y)(y_2-y_3)}
    {|y|^5}\right)\omega_1(y) \diff y
    =0,
    \end{equation}
    and that therefore
    \begin{equation}
     P.V.
    \int_{\mathbb{R}^3}
    \left(\frac{(\sigma\cdot y)(y_2-y_3)}
    {|y|^5}\right)\omega_1(y) \diff y
    =
    \int_{\mathbb{R}^3}
    \left(\frac{(\sigma\cdot y)(y_2-y_3)}
    {|y|^5}\right)\omega_1(y) \diff y,
    \end{equation}
    where the later integral is convergent.
    We can then drop the principal value and conclude that
    \begin{equation}
    \lambda
    =
    \frac{3}{8\pi}
    \int_{\mathbb{R}^3}
    \left(\frac{(\sigma\cdot y)(y_2-y_3)}
    {|y|^5}\right)\omega_1(y) \diff y.
    \end{equation}

    Again taking the change of variables $x=P_{12}(y)$
    and applying the associated permutation symmetry from \Cref{SwapPermuteCor}, we find that
    \begin{align}
    \lambda
    &=
    -\frac{3}{8\pi}\int_{\mathbb{R}^3}
    \frac{(\sigma\cdot y)(y_2-y_3)}{|y|^5}
    \omega_2(P_{12}(y)) \diff y \\
    &=
    \frac{3}{8\pi}\int_{\mathbb{R}^3}
    \frac{(\sigma\cdot x)(-x_1+x_3)}{|x|^5}
    \omega_2(x) \diff x.
    \end{align}
    The computation for \cref{LambdaEqn3}, the formula involving the third vorticity component, follows by the same methods using the change of variables 
    $x=P_{13}(y)$.
    This completes the proof.
\end{proof}

\begin{corollary} \label{GradCor}
    Suppose $u\in H^s_{df}, s>\frac{5}{2}$ is permutation symmetric.
    Then the gradient at the origin is given by
    \begin{equation}
    \nabla u(0)
    =
    \lambda 
    \left( \begin{array}{ccc}
        0 & -1 & -1  \\
        -1 & 0 & -1  \\
        -1 & -1 & 0
    \end{array}\right),
    \end{equation}
    where
    \begin{equation}
    \lambda=
    -\frac{1}{8\pi}\int_{\mathbb{R}^3}
    \left(\frac{\sigma\cdot x}
    {|x|^5}\right)(\sigma\times x)\cdot \omega(x) \diff x.
    \end{equation}
    Note this can be expressed as
    \begin{equation}
    \lambda=
    \frac{1}{8\pi}\int_{\mathbb{R}^3}
    \left(\frac{\sigma\cdot x}
    {|x|^5}\right)
    \left(\begin{array}{c}
     x_2-x_3 \\ x_3-x_1 \\ x_1-x_2 
    \end{array}\right)
    \cdot \omega(x) \diff x.
    \end{equation}
\end{corollary}

\begin{proof}
    Taking one third the sum of \cref{LambdaEqn1,LambdaEqn2,LambdaEqn3},
    the result follows immediately from \Cref{GradientOriginThm}.
\end{proof}

\begin{remark}
    Based on this identity, it might seem like a natural candidate for finite-time blowup would be to take initial data of the form
    \begin{equation}
    \omega^0(x)=
    -f\left(|x|^2\right)
    (\sigma\cdot x)
    (\sigma \times x),
    \end{equation}
    for some positive function $f$.
    Unfortunately, permutation-symmetric vector fields of this form cannot lead to finite-time blowup for smooth solutions of the Euler equation.
    This is because, as we will see in the next section, such vector fields are in the class of axisymmetric, swirl-free flows, but with the $\sigma$-axis as the axis of symmetry, rather than the $x_3$ axis.
\end{remark}

We should note that the vorticity in the $\sigma$-direction does not contribute at all to the gradient at the origin because $(\sigma\times x)\cdot \sigma=0$.
We can therefore express the gradient at the origin in terms of the part of the vorticity that is orthogonal to $\sigma$.

\begin{corollary} \label{GradCorPerp}
    Suppose $u\in H^s_{df}, s>\frac{5}{2}$ is permutation symmetric.
    Then the gradient at the origin is given by
    \begin{equation}
    \nabla u(0)
    =
    \lambda 
    \left( \begin{array}{ccc}
        0 & -1 & -1  \\
        -1 & 0 & -1  \\
        -1 & -1 & 0
    \end{array}\right),
    \end{equation}
    with
    \begin{equation}
    \lambda=
    -\frac{1}{8\pi}\int_{\mathbb{R}^3}
    \left(\frac{\sigma\cdot x}
    {|x|^5}\right)(\sigma\times x)\cdot \omega^\perp(x) \diff x,
    \end{equation}
    where
    \begin{equation}
    \omega^\perp
    =
    \omega
    -\frac{1}{3}(\omega\cdot\sigma)\sigma.
    \end{equation}
    Note that
    \begin{equation}
    \sigma\cdot\omega^\perp=0.
    \end{equation}
\end{corollary}

\begin{proof}
    The result follows immediately from \Cref{GradCor} and the fact that
    $\sigma\cdot (\sigma\times x)=0$.
\end{proof}

\subsection{Generalized axisymmetric, swirl-free flows} \label{PermuteAxisymSubsection}

Axisymmetric, swirl-free flows are generally considered with an axisymmetry about the 
$x_3$-axis. In this subsection we will consider axisymmetric flows about general axes, and we will show that axisymmetric, swirl-free flows with the $\sigma$-axis serving as the axis of symmetry are permutation symmetric. This will allow us to prove that there are permutation-symmetric $C^{1,\alpha}$ solutions of the Euler equation that blowup in finite-time by taking a rotation of the blowup solution introduced in \cites{Elgindi,ElgindiGhoulMasmoudi}.

\begin{definition} \label{AxisymDef}
    Fix a unit vector $v\in\mathbb{R}^3, |v|=1$.
    We will say that a vector field $u\in H^s\left(\mathbb{R}^3;
    \mathbb{R}^3\right), 
    s>\frac{5}{2}$,
    is $v$-axisymmetric, swirl-free 
    if
    \begin{equation}
    u(x)=u_r(r,z)e_r+ u_z(r,z)e_z,
    \end{equation}
    where
    \begin{align}
    z&=x\cdot v \\
    x'&=x-(x\cdot v)v \\
    r&=|x'| \\
    e_r &= \frac{x'}{|x'|} \\
    e_z &= v.
    \end{align}
    Note that the classical definition of an axisymmetric, swirl-free vector field takes the $z$-axis to be the axis of symmetry, corresponding in our definition to $e_3$-axisymmetric, swirl-free.
\end{definition}

\begin{proposition} \label{AxisymDivProp}
    For all $u\in H^s\left(\mathbb{R}^3;
    \mathbb{R}^3\right), s>\frac{5}{2}$,
    $v$-axisymmetric, swirl-free,
    \begin{equation}
    \nabla\cdot u=
    \partial_r u_r+ \partial_z u_z
    +\frac{1}{r}u_r.
    \end{equation}
\end{proposition}

\begin{proof}
    First we observe that
    \begin{equation}
    \nabla r= e_r,
    \end{equation}
    and that
    \begin{align}
    \nabla \cdot e_r
    &=
    \frac{\nabla\cdot x'}{r}
    -x'\cdot\nabla
    \frac{1}{r}\\
    &=
    \frac{2}{r}-r\frac{1}{r^2} \\
    &=
    \frac{1}{r}.
    \end{align}
    Applying the chain rule we find that
    \begin{align}
    \nabla \cdot u
    &=
    e_r\cdot \nabla u_r
    +e_z\cdot\nabla u_z
    +u_r \nabla\cdot e_r \\
    &=
    \partial_r u_r+ \partial_z u_z
    +\frac{1}{r}u_r,
    \end{align}
    and this completes the proof.
\end{proof}

\begin{proposition} \label{AxisymCurlPropA}
    Suppose $u\in H^s\left(\mathbb{R}^3;
    \mathbb{R}^3\right), s>\frac{5}{2}$ is $v$-axisymmetric, swirl-free.
    Then
    \begin{align}
    \omega(x)
    &=
    (\partial_z u_r-\partial_r u_z)(r,z) 
    v\times e_r \\
    &=
    \frac{\partial_z u_r-\partial_r u_z}
    {r}(r,z) v\times x
    \end{align}
\end{proposition}

\begin{proof}
    It is straightforward to compute that $e_r$ is irrotational, with
    \begin{equation}
    \nabla\times e_r=0.
    \end{equation}
    Therefore, we may conclude that
    \begin{align}
    \omega
    &=
    \partial_z u_r e_z\times e_r
    +\partial_r u_z e_r\times e_z \\
    &=
    \left(\partial_z u_r-
    \partial_r u_z\right) 
    e_z\times e_r.
    \end{align}
    Recall that $v=e_z$ and that
    $x=re_r+ze_z$, and so
    \begin{equation}
    v\times x= r e_z\times e_r,
    \end{equation}
    and this completes the proof.
\end{proof}

\begin{lemma} \label{AxisymMirrorSymLemma}
    Suppose $u\in C\left(\mathbb{R}^3;\mathbb{R}^3\right)$, 
    is $v$-axisymmetric, swirl-free,
    \begin{equation}
    u(x)=u_r(r,z)e_r+u_z(r,z)e_z.
    \end{equation}
    Then $u$ is $v$-mirror symmetric if and only if $u_r$ is even in $z$ and $u_z$ is odd in $z$.
\end{lemma}

\begin{proof}
    This is a straightforward computation left to the reader.
\end{proof}

\begin{proposition} \label{OrthoEquivProp}
    Suppose $Q\in O(3)$ has columns
    \begin{equation}
    Q=[w,\Tilde{w},v]
    \end{equation}
    satisfying
    \begin{equation}
    w\times \Tilde{w}=v.
    \end{equation}
    Then $u\in C\left(\mathbb{R}^3;\mathbb{R}^3\right)$, 
    is $v$-axisymmetric, swirl-free if and only if $u^{Q^{tr}}$ is 
    $e_3$-axisymmetric, swirl-free.

    Furthermore, $u$ is $v$-axisymmetric, swirl-free and $v$-mirror symmetric if and only if $u^{Q^{tr}}$ is $e_3$-axisymmetric, swirl-free and $e_3$-mirror symmetric.
\end{proposition}

\begin{proof}
    We begin with the forward direction.
    Suppose $u$ is $v$-axisymmetric, swirl-free, and is therefore given by
    \begin{equation}
    u(x)=u_r(r,z)e_r+u_z(r,z)e_z,
    \end{equation}
    where
    \begin{align}
    z&=x\cdot v \\
    x'&=x-(x\cdot v)v \\
    r&=|x'| \\
    e_r &= \frac{x'}{|x'|} \\
    e_z &= v.
    \end{align}
    Now let
    \begin{equation}
    \Tilde{u}(x)
    = u_r(\Tilde{r},\Tilde{z})e_{\Tilde{r}}
    +u_z(\Tilde{r},\Tilde{z})e_3
    \end{equation}
    where
    \begin{align}
    \Tilde{z}&=x_3\\
    \Tilde{x}'&=(x_1,x_2,0) \\
    \Tilde{r}&=|\Tilde{x}'| \\
    e_{\Tilde{r}} &= \frac{\Tilde{x}'}
    {|\Tilde{x}'|}.
    \end{align}
    It is clear from \Cref{AxisymDef} and \Cref{AxisymDivProp}, that $\Tilde{u}$ is $e_3$-axisymmetric, swirl-free. 
    
    It remains only to show that 
    \begin{equation}
    \Tilde{u}=u^{Q^{tr}}.
    \end{equation}
    We begin by letting
    $y=Q^{tr}x,$
    and taking $\Tilde{z},\Tilde{r},
    e_{\Tilde{r}}$ as above but in terms of $y$.
    We can then clearly see that
    \begin{equation}
    \Tilde{z}=y_3=x\cdot v=z.
    \end{equation}
    Likewise, we can see that
    \begin{align}
    \Tilde{y}'
    &=
    y-y_3e_3 \\
    &=
    Q^{tr}x- (x\cdot v)e_3 \\
    &=
    Q^{tr}(x-(x\cdot v)v) \\
    &=
    Q^{tr}x'.
    \end{align}
    We know that multiplication by orthogonal matrices is length preserving, so 
    \begin{equation}
    \Tilde{r}=|\Tilde{y}'|=|x'|=r,
    \end{equation}
    and consequently 
    \begin{equation}
    e_{\Tilde{r}}= Q^{tr} e_r.
    \end{equation}
    Applying the above identities, we can compute that
    \begin{align}
    \Tilde{u}^Q(x)
    &=
    Q\Tilde{u}\left(Q^{tr}x\right) \\
    &=
    Q\Tilde{u}(y) \\
    &=
    u_r(\Tilde{r},\Tilde{z})Qe_{\Tilde{r}}
    +u_z(\Tilde{r},\Tilde{z})Qe_3 \\
    &=
    u_r(r,z)e_r+u_z(r,z)e_z \\
    &=
    u(x).
    \end{align}
    Applying \Cref{OrthoComposeProp},
    we can see that
    \begin{equation}
    \Tilde{u}^Q=u,
    \end{equation}
    implies that
    \begin{equation}
    \Tilde{u}=u^{Q^{tr}}.
    \end{equation}

    Now suppose that additionally $u$ is $v$-mirror symmetric. Applying \Cref{AxisymMirrorSymLemma}, we can see that $u_r$ is even in $z$ and $u_z$ is odd in $z$. Again applying \Cref{AxisymMirrorSymLemma}, this implies that $u^{Q^{tr}}=\Tilde{u}$ is $e_3$-mirror symmetric, which completes the forward direction of the proof under the additional mirror symmetric hypothesis.
    
    The proof of the reverse direction is essentially the same. Just define $u$ in terms of $\Tilde{u}$, instead of the other way around, and show that
    \begin{equation}
    u=\Tilde{u}^Q.
    \end{equation}
\end{proof}

\begin{remark}
    Note that equivalently, $\Tilde{u}$ is $e_3$-axisymmetric, swirl-free if and only if $\Tilde{u}^Q$ is $v$-axisymmetric and swirl-free, with the same hypotheses on $Q$, as we can see from the construction in the proof.
    Likewise, $\Tilde{u}$ is $e_3$-axisymmetric, swirl-free and $e_3$-mirror symmetric if and only if $\Tilde{u}^Q$ is $v$-axisymmetric, swirl-free, and $v$-mirror symmetric.
\end{remark}

\begin{proposition} \label{AxisymCurlPropB}
    Suppose $u\in H^s_{df}, s>\frac{7}{2}$.
    Then $u$ is $v$-axisymmetric and swirl-free if and only if
    \begin{equation}
    \omega(x)=-\phi(r,z)v\times x,
    \end{equation}
    where $\phi\in C^\alpha\left(\mathbb{R}^+\times 
    \mathbb{R}\right), \alpha=s-\frac{7}{2}$, is given by
    \begin{equation}
    \phi(r,z)=
    \frac{-\partial_zu_r+\partial_r u_z}{r}
    \end{equation}
\end{proposition}

\begin{proof}
    We have already shown in \Cref{AxisymCurlPropA} that if $u$ is $v$-axisymmetric and swirl-free, then
    \begin{align}
    \omega
    &=
    \frac{\partial_zu_r-\partial_r u_z}{r}
    v\times x \\
    &=
    -\phi(r,z) v\times x,
    \end{align}
    and we know that $e_r$ has a singularity at when $r=0$, so in order for H\"older continuous $\nabla\omega$ to be continuous, we need
    $\frac{\partial_zu_r-\partial_r u_z}{r}$ to be H\"older continuous, and hence $\phi\in C^\alpha$.
    Now consider the other direction.
    Suppose
    \begin{equation}
    \omega(x)=-\phi(r,z)v\times x.
    \end{equation}
    Let 
    \begin{equation}
    \Tilde{u}=u^{Q^{tr}},
    \end{equation}
    where $Q$ is defined as above. Because we have $\det(Q)=1$ by construction, we can see that
    \begin{align}
    \Tilde{\omega}(x)
    &=
    -\phi(\Tilde{r},\Tilde{z})
    e_3\times x \\
    &=
    -\Tilde{r}\phi(\Tilde{r},\Tilde{z}) e_\theta.
    \end{align}
    It is classical that this implies that $\Tilde{u}$ is $e_3$-axisymmetric and swirl-free, so by \Cref{OrthoEquivProp}, we find that
    $u$ is $v$-axisymmetric and swirl-free,
    which completes the proof.
\end{proof}

\begin{theorem}
    Suppose $u^0\in H^s_{df}, s>\frac{5}{2}$ is $v$-axisymmetric, swirl-free.
    Then there exists a unique, global smooth solution of the Euler equation 
    $u\in C\left([0,+\infty);H^s_{df}\right)
    \cap C^1\left([0,+\infty);
    H^{s-1}_{df}\right)$,
    and this solution is also $v$-axisymmetric and swirl-free.
    Furthermore, if $s>\frac{7}{2}$, then for all $0\leq t<+\infty$, the vorticity can be expressed by
    \begin{equation}
    \omega(x,t)=
    -\phi(r,z,t)
    v\times x,
    \end{equation}
    where $\phi$ is transported by the flow, with
    \begin{equation}
        \partial_t\phi
        +u_r\partial_r\phi
        +u_z\partial_z\phi
        =0.
    \end{equation}
\end{theorem}

\begin{proof}
    The proof amounts to using the equivalence results we have just proven to lift global regularity for $e_3$-axisymmetric solutions of the Euler equation, which is classical, to general $v$-axisymmetric solutions, and to show that the knew relative vorticity is still transported.
    In particular, we will apply the original results of Ukhovskii and Yudovich \cite{Yudovich}, as well as the more recent results of Danchin \cite{Danchin}, which allow us to prove global regularity when we only assume $u\in H^s$, with $s>\frac{5}{2}$.
    We begin by letting 
    \begin{equation}
    \Tilde{u}^0=\left(u^0\right)^{Q^{tr}},
    \end{equation}
    with $Q$ defined as above.
    We know that $\Tilde{u}^0$ is $e_3$ axisymmetric and swirl-free, and so there exists a unique, global smooth solution of the Euler equation
    $\Tilde{u}\in C\left([0,+\infty);H^s_{df}\right)
    \cap C^1\left([0,+\infty);
    H^{s-1}_{df}\right)$.
    Now we we will let $u\in C\left([0,+\infty);H^s_{df}\right)
    \cap C^1\left([0,+\infty);
    H^{s-1}_{df}\right)$ be defined by
    \begin{equation}
    u=\Tilde{u}^Q.
    \end{equation}
    We can see therefore that $u$ is the unique solution of the Euler equation with initial data $u^0$, and that $u$ is $v$-axisymmetric and swirl-free.

    This completes the proof of global existence;
    we will now show that the relative vorticity is transported by the flow when $s>\frac{7}{2}$.
    We know that 
    \begin{equation}
    \Tilde{\omega}(x)=
    \omega_\theta(\Tilde{r},\Tilde{z})e_\theta,
    \end{equation}
    where $\omega_\theta$ must vanish linearly at $\Tilde{r}=0$.
    Defining $\phi$ by
    \begin{equation}
    \phi(\Tilde{r},\Tilde{z})=
    -\frac{\Tilde{\omega}_\theta
    (\Tilde{r},\Tilde{z})}
    {\Tilde{r}},
    \end{equation}
    we can see that
    \begin{equation}
    \Tilde{\omega}(x)=
    -\phi(\Tilde{r},\Tilde{z})e_3 \times x.
    \end{equation}
    It then follows from \Cref{VortThmSO} that
    \begin{align}
    \omega(x)
    &=
    Q \Tilde{\omega}(Q^{tr}x) \\
    &=
    -\phi(r,z) Q(e_3\times Q^{tr}x) \\
    &=
    -\phi(r,z Q 
    ((x\cdot w)e_2
    -(x\cdot\Tilde{w})e_1) \\
    &=
    -\phi(r,z) 
    ((x\cdot w) \Tilde{w}
    -(x\cdot\Tilde{w})w) \\
    &=
    -\phi(r,z) 
    ((x\cdot w) (v\times w)
    +(x\cdot\Tilde{w})(v\times \Tilde{w})) \\
    &=
    -\phi(r,z) v\times x.
    \end{align}
    
    It now remains only to show that $\phi$ is transported by the flow.
    It is classical that, for $e_3$-axisymmetric, swirl-free solutions of the Euler equation, the vorticity equation can be expressed by
    \begin{equation}
    \partial_t\Tilde{\omega}_\theta
    +\Tilde{u}_r\partial_{\Tilde{r}}
    \Tilde{\omega}_\theta
    +\Tilde{u}_z\partial_z\Tilde{\omega}_\theta
    -\frac{\Tilde{u}_r}
    {\Tilde{r}}\Tilde{\omega}_\theta
    =0,
    \end{equation}
    and that therefore
    \begin{equation}
    \left(\partial_t
    +\Tilde{u}_r\partial_{\Tilde{r}}
    +\Tilde{u}_z\partial_{\Tilde{z}}\right)
    \left(-\frac{\Tilde{\omega}_\theta}
    {\Tilde{r}}\right)
    =0,
    \end{equation}
    and so $\phi$ is transported by the flow. This completes the proof and shows that we have the correct generalized definition of relative vorticity.
\end{proof}

\begin{proposition} \label{AxisymPermuteProp}
    Suppose $u\in L^2 \cap C^{1,\alpha}, \alpha>0$ is $\Tilde{\sigma}$-axisymmetric, swirl-free. Then $u$ is permutation symmetric.
\end{proposition}

\begin{proof}
    First we observe that if $u$ is $\Tilde{\sigma}$-axisymmetric, swirl-free, then
    \begin{equation}
    u(x)=u_r(r,z)e_r+ u_z(r,z) \Tilde{\sigma},
    \end{equation}
    where
    \begin{align}
    z&=
    \frac{1}{\sqrt{3}}\sigma\cdot x \\
    x'&=x-\frac{1}{3}(\sigma\cdot x)\sigma \\
    r&=|x'|
    =\left(|x|^2-\frac{1}{3}(\sigma\cdot x)^2
    \right)^\frac{1}{2} \\
    e_r&=\frac{x'}{r}.
    \end{align}
    Observe that $r,z,\Tilde{\sigma}$ are all invariant under permutations. Fix a permutation $P\in\mathcal{P}_3$.
    Letting $\Tilde{x}=P^{tr}x$,
    it is immediate that
    \begin{align}
    \Tilde{z}&=z \\
    \Tilde{r} &= r \\
    \Tilde{x}'&=P^{tr}x' \\
    e_{\Tilde{r}}&= P^{tr} e_r.
    \end{align}
    Next we compute that
    \begin{align}
    u^P(x)
    &=
    Pu(P^{tr}x) \\
    &=
    Pu(\Tilde{x}) \\
    &=
    P\left(
    u_r(\Tilde{r},\Tilde{z})e_{\Tilde{r}}
    +u_z(\Tilde{r},\Tilde{z})\Tilde{\sigma}
    \right) \\
    &=
    P \left( u_r(r,z)P^{tr}e_r
    +u_z(r,z) \Tilde{\sigma}\right) \\
    &=
    u_r(r,z)e_r+ u_z(r,z) \Tilde{\sigma} \\
    &=
    u(x).
    \end{align}
    Therefore we can see that for all $P\in\mathcal{P}_3$,
    \begin{equation}
    u=u^P,
    \end{equation}
    which completes the proof.
\end{proof}

We can now prove \Cref{AlphaPermuteBlowupIntro}, which is restated here for the reader's convenience.

\begin{theorem} \label{AlphaPermuteBlowup}
    There exists $\alpha>0$ and an odd, permutation symmetric, $\sigma$-mirror symmetric solution of the incompressible Euler equation
    $u\in C\left([0,1);L^2\cap C^{1,\alpha}\right)
    \cap C^1\left([0,1);L^2\cap C^\alpha\right)$,
    satisfying
    \begin{equation}
    \int_0^1 \|\omega(\cdot,t)\|_{L^\infty}
    \diff t
    =+\infty,
    \end{equation}
    and for all $0\leq t <1$,
    \begin{equation}
    \|u(\cdot,t)\|_{L^2}^2
    =\left\|u^0\right\|_{L^2}^2.
    \end{equation}
\end{theorem}

\begin{proof}
    Let
    $\Tilde{u} \in C\left([0,1);L^2\cap C^{1,\alpha}\right)
    \cap C^1\left([0,1);L^2\cap C^\alpha\right)$
    be the $e_3$-axisymmetric, swirl-free, $e_3$-mirror symmetric solution of the Euler equation that Eligndi \cite{Elgindi} and Elgindi, Ghoul, and Masmoudi \cite{ElgindiGhoulMasmoudi} proved blows up in finite-time.
    This solution satisfies the energy equality
    \begin{equation}
    \|\Tilde{u}(\cdot,t)\|_{L^2}^2
    =
    \left\|\Tilde{u}^0\right\|_{L^2}^2.
    \end{equation}
    and the Beale-Kato-Majda criterion.
    \begin{equation}
    \int_0^1 \|\Tilde{\omega}(\cdot,t)
    \|_{L^\infty} \diff t
    =+\infty.
    \end{equation}
    Define $Q\in O(3)$ by
    \begin{equation}
    Q=
    \left(\begin{array}{ccc}
        \frac{1}{\sqrt{2}} & 
        \frac{1}{\sqrt{6}} &
        \frac{1}{\sqrt{3}} \\
        -\frac{1}{\sqrt{2}} &
        \frac{1}{\sqrt{6}} &
         \frac{1}{\sqrt{3}} \\
         0 &
         -\frac{2}{\sqrt{6}} &
         \frac{1}{\sqrt{3}},
    \end{array}
    \right)
    \end{equation}
    and let
    \begin{equation}
    u=\Tilde{u}^Q.
    \end{equation}
    We can see that $u\in C\left([0,1);L^2\cap C^{1,\alpha}\right)
    \cap C^1\left([0,1);L^2\cap C^\alpha\right)$ is a solution of the Euler equation. Applying \Cref{OrthoEquivProp}, we can conclude that $u$ is $\Tilde{\sigma}$-axisymmetric, swirl-free, and $\Tilde{\sigma}$-mirror symmetric. 
    Applying \Cref{AxisymPermuteProp}, we can see that $u$ is permutation symmetric.
    Because $\det(Q)=1$, we can conclude that
    \begin{equation}
    \omega=\Tilde{\omega}^Q.
    \end{equation}
    It then immediately follows from the volume preserving properties of $O(3)$ that
    for all $0\leq t <1$,
    \begin{equation}
    \|u(\cdot,t)\|_{L^2}^2
    =\left\|u^0\right\|_{L^2}^2,
    \end{equation}
    and that
    \begin{equation}
    \int_0^1 \|\omega(\cdot,t)\|_{L^\infty}
    \diff t
    =+\infty.
    \end{equation}
    This completes the proof.
\end{proof}

\subsection{Further geometric considerations}

The planes $x_1=x_2, \: x_1=x_3, \: x_2=x_3$ and the axis $x_1=x_2=x_3$ are important objects for permutation symmetric flows, because these three planes are each invariant for one of the swap permutations, and this axis is invariant under all permutations. In this subsection, we describe the velocity and the vorticity in these subspaces of $\mathbb{R}^3$ when studying permutation symmetric vector fields.

\begin{proposition} \label{PlanesOfSymProp}
    Suppose $u\in H^s_{df}, s>\frac{3}{2}$ is permutation symmetric.
    Then for all $x\in\mathbb{R}^3, x_1=x_2$,
    \begin{equation}
    u_1(x)=u_2(x),
    \end{equation}
    for all $x\in\mathbb{R}^3, x_1=x_3$,
    \begin{equation}
    u_1(x)=u_3(x),
    \end{equation}
    and for all $x\in\mathbb{R}^3, x_2=x_3$,
    \begin{equation}
    u_2(x)=u_3(x),
    \end{equation}
\end{proposition}

\begin{proof}
    Suppose $x_1=x_2$ and $u$ is permutation symmetric. Then $x=P_{12}(x)$,
    and so
    \begin{equation}
    u(x)=P_{12}u(x),
    \end{equation}
    and so
    \begin{equation}
    u_1(x)=u_2(x).
    \end{equation}
    The cases $x_1=x_3$ and $x_2=x_3$ are entirely analogous using $P_{13}$ and $P_{23}$ respectively, and are left to the reader.
\end{proof}

\begin{proposition} \label{PlanesOfSymVortProp}
    Suppose $u\in H^s_{df}, s>\frac{5}{2}$ is permutation symmetric.
    Then for all $x\in\mathbb{R}^3, x_1=x_2$,
    \begin{align}
    \omega_1(x)&=-\omega_2(x) \\
    \omega_3(x)&=0,
    \end{align}
    for all $x\in\mathbb{R}^3, x_1=x_3$,
    \begin{align}
    \omega_1(x)&=-\omega_3(x) \\
    \omega_2(x)&=0,
    \end{align}
    and for all $x\in\mathbb{R}^3, x_2=x_3$,
    \begin{align}
    \omega_2(x)&=-\omega_3(x) \\
    \omega_1(x)&=0,
    \end{align}
\end{proposition}

\begin{proof}
    Suppose $x_1=x_2$ and $u$ is permutation symmetric. Then $x=P_{12}(x)$,
    and so applying \Cref{VortPermSymThm},
    we can see that
    \begin{equation}
    \omega(x)=-P_{12}\omega(x),
    \end{equation}
    that is
    \begin{equation}
    \left(\begin{array}{c}
        \omega_1(x) \\ \omega_2(x) \\ \omega_3(x)
    \end{array}\right)
    =
    \left(\begin{array}{c}
        -\omega_2(x) \\ -\omega_1(x) \\ -\omega_3(x)
    \end{array}\right),
    \end{equation}
    and therefore
    \begin{align}
    \omega_1(x)&=-\omega_2(x) \\
    \omega_3(x)&=0.
    \end{align}
    The cases $x_1=x_3$ and $x_2=x_3$ are entirely analogous using $P_{13}$ and $P_{23}$ respectively, and are left to the reader.
\end{proof}

\begin{corollary}
    Suppose $u\in H^s_{df}, s>\frac{3}{2}$ is permutation symmetric. Then for all 
    $x\in\mathbb{R}^3, x_1=x_2=x_3$,
    \begin{equation}
    u_1(x)=u_2(x)=u_3(x),
    \end{equation}
    and so
    \begin{equation}
    u(x)=u_1(x)\sigma.
    \end{equation}
\end{corollary}

\begin{proof}
    This follows immediately from \Cref{PlanesOfSymProp}
\end{proof}

\begin{corollary}
    Suppose $u\in H^s_{df}, s>\frac{5}{2}$ is permutation symmetric. Then for all 
    $x\in\mathbb{R}^3, x_1=x_2=x_3$,
    \begin{equation}
    \omega(x)=0.
    \end{equation}
\end{corollary}

\begin{proof}
    This follows immediately from \Cref{PlanesOfSymVortProp}
\end{proof}

\begin{remark}
    When considering the single-component vorticity equation in terms of $\omega_1$, \Cref{PlanesOfSymProp,PlanesOfSymVortProp} imply that the plane $x_2=x_3$ is preserved by the flow map, and that $\omega_1=0$ in this plane. One approach, based on this fact, could be to treat 
    \begin{equation}
    \zeta= 
    \frac{\omega_1}{x_2-x_3},
    \end{equation}
    as the generalization of the relative vorticity for permutation-symmetric solutions, noting that this quantity will be bounded for smooth solutions. Recall that in the axisymmetric case (about the axis $x_1=x_2=x_3)$,
    \begin{equation}
    \omega(x)=\phi(r,z)
    \left(\begin{array}{c}
          x_2-x_3 \\ x_3-x_1 \\ x_1-x_2,
    \end{array}\right)
    \end{equation}
    and so 
    \begin{equation}
    \phi(r,z)=\frac{\omega_1(x)}{x_2-x_3}.
    \end{equation}
    So $\zeta$ is simply the relative vorticity in the special case of $\Tilde{\sigma}$-axisymmetric, swirl-free.
    Note that $\zeta$ remains well defined for smooth, permutation-symmetric solutions that are not axisymmetric, swirl-free with respect to the $\sigma$-axis. In the non-axisymmetric case, however, the quantity $\zeta$ will not be transported by the flow, and therefore has the potential to unlock insights into finite-time blowup.
\end{remark}

\section{Permutation symmetry in cylindrical coordinates}

We are interested in the behaviour of $\mathcal{G}_\sigma$-symmetric solutions of the Euler equation, because $\mathcal{G}_\sigma$-symmetric solutions of the Fourier-restricted Euler model equation---which has many of the key features of the full Euler equation---blowup in finite-time \cite{MillerRestricted}.
In this section, we will prove \Cref{EquivThmIntro}, showing the the symmetry groups $\mathcal{G}$ and $\mathcal{G}_\sigma$ are equivalent up to rotation.
The change of coordinates by the rotation $Q_\sigma$ allows us to recast the discrete group of symmetries $\mathcal{G}_\sigma$ as $\mathcal{G}$. 
This is useful, because $\mathcal{G}$ can be expressed in terms of mirror symmetries in the coordinate axes and rotational symmetries in the horizontal plane, which lends itself naturally to expression in cylindrical coordinates as a Fourier series in the $\theta$ variable.

\subsection{Basic properties of the symmetry groups
$\mathcal{G}$ and $\mathcal{G}_\sigma$}

We begin by considering the main properties of the groups $\mathcal{G}_\sigma$ and $\mathcal{G}$, in particular their generators.

\begin{lemma} \label{SigmaCommuteLemma}
    The reflection $M_{\Tilde{\sigma}}$ commutes with $\mathcal{P}_3$; for all $P\in\mathcal{P}_3$
    \begin{equation}
    M_{\Tilde{\sigma}}P=
    PM_{\Tilde{\sigma}}.
    \end{equation}
\end{lemma}

\begin{proof}
    This follows immediately from the fact that for all $P\in\mathcal{P}_3, x\in\mathbb{R}^3,$
    \begin{equation}
    \Tilde{\sigma}\cdot x= \Tilde{\sigma} \cdot P(x)
    \end{equation}
\end{proof}

\begin{proposition} \label{PermuteGenerateProp}
    The symmetry group $\mathcal{P}_3$ is generated by $P_{12}$ and $P_f$, and the six elements of $\mathcal{P}_3$
    are given by
    \begin{equation}
    P=P_{12}^\alpha P_f^\beta,
    \end{equation}
    where $\alpha\in\{0,1\}, 
    \beta\in\{0,1,2\}$.
\end{proposition}

\begin{proof}
    Observe that
    \begin{align}
    P_b&= P_f^2 \\
    P_{23} &= P_{12} P_f \\
    P_{13} &= P_{12} P_b = P_{12}P_f^2,
    \end{align}
    and this completes the proof.
\end{proof}

\begin{proposition} \label{SigmaGenerateProp1}
    The symmetry group $\mathcal{G}_\sigma$ is generated by $-I_3,M_\sigma,P_{12},P_f$. 
    Furthermore, $\mathcal{G}_\sigma$ has the 24 elements
    \begin{equation}
    \mathcal{G}_\sigma
    =
    \pm \mathcal{P}_3 \cup 
    \pm M_{\Tilde{\sigma}} \mathcal{P}_3;
    \end{equation}
    that is, for all $Q\in\mathcal{G}_\sigma$
    \begin{equation}
    Q=
    (-1)^a M_{\Tilde{\sigma}}^b P_{12}^\alpha P_f^\beta,
    \end{equation}
    where $a,b,\alpha\in \{0,1\}$ and 
    $\beta\in\{0,1,2\}$.
\end{proposition}

\begin{proof}
    We know by definition that any element in $\mathcal{G}_\sigma$ can be written as an arbitrary product of $-I_3, M_{\Tilde{\sigma}}$ and elements of $\mathcal{P}_3$. We know that $M_\sigma$ commutes with all permutations, and that $-I_3$ commutes with all matrices. The product of permutations is again a permutation, and so we can see that for all $Q\in \mathcal{G}_\sigma$,
    \begin{align}
    Q
    &=(-I_3)^a M_{\Tilde{\sigma}}^b P \\
    &= (-1)^a M_{\Tilde{\sigma}}^b P,
    \end{align}
    for some $a,b\in\mathbb{N}, P\in\mathcal{P}_3$.
    We can take $a,b\in\{0,1\}$ without loss of generality, and applying \Cref{PermuteGenerateProp} this completes the proof.
\end{proof}

\begin{proposition} \label{SigmaGenerateProp2}
    The symmetry group $\mathcal{G}_\sigma$ is generated by $M_\sigma, P_{12}, -M_\sigma P_b$. 
\end{proposition}

\begin{proof}
    It is enough to prove that these three elements generate $-I_3,M_{\Tilde{\sigma}},P_{12},P_f$, and then we can apply \Cref{SigmaGenerateProp1}.
    We clearly have symmetries $P_{12}$ and $M_{\Tilde{\sigma}}$ by hypothesis.
    Observe that 
\begin{align}
    (-M_{\Tilde{\sigma}}P_b)^2
    &=
    M_{\Tilde{\sigma}}^2 P_b^2 \\
    &=
    P_f,
\end{align}
    so $P_f$ is generated by the three elements.
    Finally observe that
    \begin{align}
    M_{\Tilde{\sigma}}
    (-M_{\Tilde{\sigma}}P_b)^3
    &=
    -M_{\Tilde{\sigma}}^2 P_b P_f \\
    &=
    -I_3.
    \end{align}
    Applying \Cref{SigmaGenerateProp1}, this completes the proof.
\end{proof}

\begin{proposition} \label{GeneratePropN}
    Let $\mathcal{N}$ be the group generated by $M_{e_1},R_{\frac{2\pi}{3}}$.
    This group has six elements and is given by
    \begin{equation}
    \mathcal{N}=
    \mathcal{R}_3
    \cup M_{e_1}\mathcal{R}_3
    \end{equation}
    that is, for all $Q\in\mathcal{N}$,
    \begin{equation}
      Q=
      M_{e_1}^\alpha 
    R_{\frac{2\pi}{3}}^\beta,
    \end{equation}
    where $\alpha\in \{0,1\}, \beta\in\{0,1,2\}$.
\end{proposition}

\begin{proof}
    This follows more or less immediately from the commutator relation
    \begin{align}
    M_{e_1} R_\frac{2\pi}{3}
    &=
    R_\frac{4\pi}{3} M_{e_1}  \\
    M_{e_1} R_\frac{4\pi}{3}
    &=
    R_\frac{2\pi}{3} M_{e_1}.
    \end{align}
    This fact means that every element involving the product of one reflection and one rotation can be written with the reflection on the right, and also that any product can be reduced to one of this form, because the product of three or more elements can always be reduced to a product of two or less. For example,
    \begin{equation}
    M_{e_1}R_\frac{2\pi}{3}M_{e_1}
    =R_\frac{4\pi}{3} M_{e_1}^2
    =R_\frac{4\pi}{3}.
    \end{equation}
\end{proof}

\begin{proposition} \label{GenerateProp1}
    The symmetry group $\mathcal{G}$ is generated by
    $-I_3, M_{e_1}, M_{e_3}, R_{\frac{2\pi}{3}}$.
    Furthermore, $\mathcal{G}$ has the 24 elements
    \begin{equation}
    \mathcal{G}
    =
    \pm \mathcal{N} \cup 
    \pm M_{\Tilde{\sigma}} \mathcal{N};
    \end{equation}
    that is, for all $Q\in\mathcal{G}$,
    \begin{equation}
    Q=(-1)^a M_{e_3}^b M_{e_1}^\alpha 
    R_{\frac{2\pi}{3}}^\beta,
    \end{equation}
    where $a,b,\alpha\in \{0,1\}$ and 
    $\beta\in\{0,1,2\}$.
\end{proposition}

\begin{proof}
    By definition, $\mathcal{G}$ is generated by $M_{e_1},M_{e_2},M_{e_3}, R_\frac{\pi}{3}$,
    so to prove that $\mathcal{G}$ is generated by the four elements, it is sufficient to show that these elements generate the generators.
    Observe that
    \begin{equation}
    M_{e_2}=-M_{e_1}M_{e_3}.
    \end{equation}
    Next observe that 
    \begin{equation}
    R_\pi=R_{-\pi}=M_{e_1} M_{e_2},
    \end{equation}
    and that
    \begin{equation}
    R_{\frac{\pi}{3}}=R_{\frac{4\pi}{3}}R_{-\pi},
    \end{equation}
    and therefore, the four elements generate $\mathcal{G}$.
    Next observe that $M_{e_3}$ commutes with $M_{e_1}$ and $R_{\frac{\pi}{3}}$, so analogously to the proof of \Cref{SigmaGenerateProp1}, we can see that for all $Q\in\mathcal{G}$,
    \begin{align}
    Q
    &=(-I_3)^a M_{e_3}^b N \\
    &= (-1)^a M_{e_3}^b N,
    \end{align}
    for some $a,b\in\mathbb{N}, N\in\mathcal{N}$.
    We can take $a,b\in\{0,1\}$ without loss of generality, and applying \Cref{GeneratePropN} this completes the proof.
\end{proof}

\begin{proposition} \label{GenerateProp2}
    The symmetry group $\mathcal{G}$ is generated by
    $M_{e_1},M_{e_3},R_{\frac{\pi}{3}}$.
\end{proposition}

\begin{proof}
    It suffices to show that $M_{e_2}$ is generated by these three elements. Observe that
    \begin{equation}
    R_\frac{\pi}{3}^3=R_{\pi}= M_{e_1}M_{e_2},
    \end{equation}
    and we can conclude that
    \begin{equation}
    M_{e_2}=M_{e_1} R_\frac{\pi}{3}^3.
    \end{equation}
    This completes the proof.
\end{proof}

\subsection{Equivalence of $\mathcal{G}$ and $\mathcal{G}_\sigma$ up to rotation}

We have now characterized $\mathcal{G}$ and $\mathcal{G}_\sigma$ sufficiently to prove that these symmetry groups are equivalent up to rotation in this subsection.

\begin{proposition} \label{CoordTransSymProp}
    Suppose that $Q,M\in O(3)$
    and that $u\in C\left(\mathbb{R}^3;
    \mathbb{R}^3\right)$, and let $v=u^Q$.
    Then 
    \begin{equation}
        u^M=u
    \end{equation}
     if and only if
     \begin{equation}
         v^{QMQ^{tr}}=v.
     \end{equation}
    In other words, $M$ is a symmetry of $u$ if and only if 
    $QMQ^{tr}$ is a symmetry of $u^Q$.
\end{proposition}

\begin{proof}
    Suppose $u^M=u$.
    Recalling the identity in \Cref{OrthoComposeProp}, we can conclude that
    \begin{align}
    v^{QM Q^{tr}}
    &=
    \left(u^Q\right)^{QM Q^{tr}} \\
    &=
    u^{QM Q^{tr}Q} \\
    &=
    u^{QM} \\
    &=
    \left(u^M\right)^Q \\
    &=
    u^Q \\
    &=
    v.
    \end{align}
    Now suppose that $v^{QM Q^{tr}}=v$.
    Noting that $u=v^{Q^{tr}}$,
    we can likewise compute that
    \begin{align}
    u^M
    &=
    \left(v^{Q^{tr}}\right)^M \\
    &=
    v^{MQ^{tr}} \\
    &=
    v^{Q^{tr}QMQ^{tr}} \\
    &=
    \left(v^{QMQ^{tr}}\right)^{Q^{tr}} \\
    &=
    v^{Q^{tr}} \\
    &= 
    u.
    \end{align}
    This completes the proof.
\end{proof}

We will now note some useful matrix identities, whose computations are left to the reader.

\begin{lemma} \label{ConjugateLemma}
    The transformations $P_{12},P_f,M_{\Tilde{\sigma}}$ are equivalent to the transformations $M_{e_1},R_\frac{2\pi}{3},M_{e_3}$ under a change of coordinates by the rotation
    \begin{equation}
        Q_\sigma
        =
    \left(\begin{array}{ccc}
\frac{1}{\sqrt{2}} & \frac{1}{\sqrt{6}} 
& \frac{1}{\sqrt{3}} \\
-\frac{1}{\sqrt{2}} & \frac{1}{\sqrt{6}} 
& \frac{1}{\sqrt{3}} \\
0 & -\frac{2}{\sqrt{6}} 
& \frac{1}{\sqrt{3}} \\
    \end{array}\right).
    \end{equation}
    In particular,
    \begin{align}
    M_{e_1}&=
    Q_\sigma^{tr} P_{12} Q_\sigma \\
    R_\frac{2\pi}{3}&=
    Q_\sigma^{tr} P_f Q_\sigma \\
    M_{e_3} &=
    Q_\sigma^{tr} M_{\Tilde{\sigma}} Q_\sigma,
    \end{align}
    and equivalently
    \begin{align}
    P_{12}&=
    Q_\sigma M_{e_1} Q_\sigma^{tr} \\
    P_f&=
    Q_\sigma R_\frac{2\pi}{3}Q_\sigma^{tr} \\
    M_{\Tilde{\sigma}}&=
    Q_\sigma M_{e_3}Q_\sigma^{tr}.
    \end{align}
\end{lemma}

\begin{proposition} \label{GroupEquivProp1}
    For all
    $V\in\mathcal{G}_\sigma$ there exists $V'\in\mathcal{G}$ such that
    \begin{equation}
    V= Q_\sigma V' Q_\sigma^{tr}.
    \end{equation}
    Note that this implies that $Q_\sigma^{tr} V  Q_\sigma\in \mathcal{G}$.
\end{proposition}

\begin{proof}
    Suppose that $V\in \mathcal{G}_\sigma$.
    Applying \Cref{SigmaGenerateProp1}, we can see that
    \begin{equation}
    V=
    (-1)^a M_{\Tilde{\sigma}}^b P_{12}^\alpha P_f^\beta,
    \end{equation}
    where $a,b,\alpha\in \{0,1\}$ and 
    $\beta\in\{0,1,2\}$.
    Applying \Cref{ConjugateLemma},
    we compute that
    \begin{align}
    V&=
    (-1)^a
    \left(Q_\sigma M_{e_3}Q_\sigma^{tr}\right)^b
    \left(Q_\sigma M_{e_1} Q_\sigma^{tr}\right)^\alpha
    \left(Q_\sigma R_\frac{2\pi}{3}
    Q_\sigma^{tr}\right)^\beta \\
    &=
    Q_\sigma
    (-1)^a M_{e_3}^b M_{e_1}^\alpha 
    R_{\frac{2\pi}{3}}^\beta
    Q_\sigma^{tr},
    \end{align}
    and so the result holds with
    \begin{equation}
    V'=
    (-1)^a M_{e_3}^b M_{e_1}^\alpha 
    R_{\frac{2\pi}{3}}^\beta.
    \end{equation}
    Note that multiplying by $Q_\sigma$ on the left and $Q_\sigma^{tr}$ on the right yields
    \begin{equation}
    Q_\sigma V Q_\sigma^{tr}
    =
    V'
    \in \mathcal{G},
    \end{equation}
    completing the proof.
\end{proof}

\begin{proposition} \label{GroupEquivProp2}
    For all $V\in\mathcal{G}$ there exists $V'\in\mathcal{G}_\sigma$ such that
    \begin{equation}
    V= Q_\sigma^{tr}V' Q_\sigma.
    \end{equation}
    Note that this implies 
    $Q_\sigma V Q_\sigma^{tr}\in \mathcal{G}_\sigma$.
\end{proposition}

\begin{proof}
    The proof is exactly analogous to \Cref{GroupEquivProp1}, and is left to the reader.
\end{proof}

\begin{remark}
    Note that combining \Cref{GroupEquivProp1,GroupEquivProp2}, we can see that $\mathcal{G}$ and $\mathcal{G}_\sigma$ are equivalent up to conjugation by the rotation $Q_\sigma$:
    \begin{align}
        \mathcal{G}_\sigma
        &=
        Q_\sigma \mathcal{G} Q_\sigma^{tr} \\
        \mathcal{G}
        &=
        Q_\sigma^{tr}\mathcal{G}_\sigma Q_\sigma.
    \end{align}
    This implies that $\mathcal{G}$-symmetric and $\mathcal{G}_{\sigma}$-symmetric vector fields are equivalent up to rotation.
\end{remark}
We will now prove \Cref{EquivThmIntro}, which is restated for the reader's convenience.

\begin{theorem} \label{EquivThm}
    Suppose $u\in C\left(\mathbb{R}^3;
    \mathbb{R}^3\right)$.
    Then $u$ is $\mathcal{G}$-symmetric if and only if $u^{Q_\sigma}$ is $\mathcal{G}_\sigma$-symmetric.
\end{theorem}

\begin{proof}
    Suppose that $u$ is $\mathcal{G}$-symmetric. Then applying \Cref{CoordTransSymProp}, we can see that $Q_\sigma MQ_\sigma^{tr}$ is a symmetry of $u^{Q_\sigma}$ for all $M\in\mathcal{G}$.
    Applying \Cref{GroupEquivProp1}, we know that every element in $M'\in \mathcal{G}_\sigma$ can be written in the form $M'=Q_\sigma MQ_\sigma^{tr}$, so we can conclude that $M'$ is a symmetry of $u^{Q_\sigma}$ for all $M'\in\mathcal{G}_\sigma$.

    Now suppose that $u^{Q_\sigma}$ is $\mathcal{G}_\sigma$-symmetric. Clearly we have
    $u=\left(u^{Q_\sigma}\right)^{Q_\sigma^{tr}}$,
    so again applying \Cref{CoordTransSymProp}, 
    we find that $Q_{\sigma}^{tr}M' Q_\sigma$ is a symmetry of $u$ for all $M'\in \mathcal{G}_\sigma$.
    Applying \Cref{GroupEquivProp2}, we know that every element in $M \in \mathcal{G}$ can be written in the form $M=Q_\sigma^{tr} M' Q_\sigma$, so we can conclude that $M$ is a symmetry of $u$ for all $M\in\mathcal{G}$.
\end{proof}

\subsection{Conditions for symmetry}

In this subsection, we will give some necessary and sufficient conditions for a vector field to be $\mathcal{G}$-symmetric or $\mathcal{G}_\sigma$-symmetric.

\begin{lemma} \label{GenerateLemma}
    Suppose a finite group $\mathcal{G}\subset O(3)$ is generated by $G_1,G_2,..., G_n$. Then $u\in C\left(\mathbb{R}^3;\mathbb{R}^3\right)$ is $\mathcal{G}$-symmetric if and only if $G_1,G_2,...,G_n$ are symmetries of $u$ with
    \begin{equation}
    u=u^{G_1}=u^{G_2}=...=u^{G_n}
    \end{equation}
\end{lemma}

\begin{theorem} \label{GenerateThm}
    A vector field $u\in C\left(\mathbb{R}^3;\mathbb{R}^3\right)$ is $\mathcal{G}$-symmetric if and only if 
    $R_\frac{\pi}{3},M_{e_1},M_{e_3}$ are symmetries of $u$, that is
    \begin{equation}
    u = u^{R_\frac{\pi}{3}}
    = u^{M_{e_1}}
    = u^{M_{e_3}}.
    \end{equation}
\end{theorem}

\begin{proof}
    This follows immediately from \Cref{GenerateLemma} and \Cref{GenerateProp2}.
\end{proof}

\begin{theorem} \label{SigmaGenerateThm}
    A vector field $u\in C\left(\mathbb{R}^3;\mathbb{R}^3\right)$ is $\mathcal{G}_\sigma$-symmetric if and only if 
    $-M_{\Tilde{\sigma}} P_b,P_{12},M_{\Tilde{\sigma}}$ are symmetries of $u$, that is
    \begin{equation}
    u = 
    u^{-M_{\Tilde{\sigma}} P_b}
    = u^{M_{\Tilde{\sigma}}}
    = u^{P_{12}}.
    \end{equation}
\end{theorem}

\begin{proof}
    This follows immediately from \Cref{GenerateLemma} and \Cref{SigmaGenerateProp2}.
\end{proof}

For divergence free vector fields, we can also characterize symmetries in terms of the vorticity, including $\mathcal{G}$-symmetric and $\mathcal{G}_\sigma$-symmetric vector fields.

\begin{corollary} \label{GenerateCor}
    Suppose $u\in H^s_{df}, s>\frac{5}{2}$.
    Then $u$ is $\mathcal{G}$-symmetric if and only if
    \begin{align}
    \omega
    &=
    \omega^{R_\frac{\pi}{3}} \\
    &=
    -\omega^{M_{e_1}} \\
    &=
    -\omega^{M_{e_3}}.
    \end{align}
\end{corollary}

\begin{proof}
    This follows immediately from 
    \Cref{VortGeneralSymThm,GenerateThm}.
\end{proof}

\begin{corollary} \label{SigmaGenerateCor}
    Suppose $u\in H^s_{df}, s>\frac{5}{2}$.
    Then $u$ is $\mathcal{G}_\sigma$-symmetric 
    if and only if
    \begin{align}
    \omega
    &=
    \omega^{-M_{\Tilde{\sigma}}P_b} \\
    &=
    -\omega^{M_{P_{12}}} \\
    &=
    -\omega^{M_{\Tilde{\sigma}}}.
    \end{align}
\end{corollary}

\begin{proof}
    This follows immediately from 
\Cref{VortGeneralSymThm,SigmaGenerateThm}.
\end{proof}

\subsection{Fourier series expansion}

In this subsection, we will prove \Cref{FourierSeriesIntro} and \Cref{FourierSeriesCorIntro}, giving a Fourier series expansion for $\mathcal{G}$-symmetric solutions of the Euler equation.

\begin{theorem} \label{FourierSeries}
    Suppose $u\in H^s\left(\mathbb{R}^3;
    \mathbb{R}^3\right), s>\frac{5}{2}$.
    Then $u$ is $\mathcal{G}$-symmetric and divergence free if and only if it can be expressed as a Fourier series in cylindrical coordinates by
    \begin{multline} \label{Fourier}
    u(x)=
    u_{r,0}(r,z)e_r
    +u_{z,0}(r,z)e_z \\
    +\sum_{n=1}^{\infty}
    u_{r,n}(r,z)\cos(6n\theta)e_r
    +u_{z,n}(r,z)\cos(6n\theta)e_z
    +u_{\theta,n}(r,z)
    \sin(6n\theta)e_\theta,
    \end{multline}
    where for all $n\in\mathbb{Z}^+, u_{r,n},u_{\theta,n}$ are even in $z$
    and $u_{z,n}$ is odd in $z$,
    and furthermore, due to the divergence free constraint,
    \begin{equation} \label{DivFreeBase}
    \partial_r u_{r,0}
    +\frac{1}{r}u_{r,0}
    +\partial_z u_{z,0}
    =0,
    \end{equation}
    and for all $n\in\mathbb{N}$,
    \begin{equation} \label{DivFreeMode}
    \partial_r u_{r,n}
    +\frac{1}{r}u_{r,n}
    +\partial_z u_{z,n}
    +\frac{6n}{r}u_{\theta,n}
    =0.
    \end{equation}
    Note that this means that $u_{\theta,n}$ is completely determined by $u_{r,n}, u_{z,n}$ with
    \begin{align}
    u_{\theta,n}
    &=
    -\frac{r}{6n}
    \left(\partial_r u_{r,n}
    +\frac{1}{r}u_{r,n}
    +\partial_z u_{z,n}\right) \\
    &=
    -\frac{1}{6n}
    \left(\partial_r(ru_{r,n})
    +\partial_z(ru_{z,n})\right).
    \end{align}
\end{theorem}

\begin{proof}
    Suppose $u$ is $\mathcal{G}$-symmetric. We begin by writing $u$ in cylindrical coordinates
    \begin{equation}
    u(x)=
    u_r(r,z,\theta)e_r
    +u_z(r,z,\theta)e_z
    +u_\theta(r,z,\theta)e_\theta
    \end{equation}
    Observe that in cylindrical coordinates
    \begin{equation}
    M_{e_2}(r,z,\theta)=(r,z,-\theta),
    \end{equation}
    and applying $M_{e_2}$ to each of the coordinate vectors we find
    \begin{align}
    e_r^{M_{e_2}}&=
    M_{e_2}\left(\begin{array}{c}
         \cos(\theta)  \\ -\sin(\theta) \\ 0
    \end{array}\right)
    = e_r \\
    e_\theta^{M_{e_2}}&=
    M_{e_2}\left(\begin{array}{c}
         \sin(\theta)  \\ \cos(\theta) \\ 0
    \end{array}\right)
    =-e_\theta,
    \end{align}
    and of course $e_z=e_3$ is a constant vector and is unaffected by this transformation.
    Using the fact that $M_{e_2}$ is a symmetry of $u$,
    we can conclude that
    \begin{equation}
    u(x)= u^{M_{e_2}}(x)
    =
    u_r(r,z,-\theta)e_r
    +u_z(r,z,-\theta)e_z
    -u_\theta(r,z,-\theta)e_\theta,
    \end{equation}
    and so 
    $u_r$ and $u_z$ are even in $\theta$, and $u_\theta$ is odd in $\theta$. 
    We can also see that $e_r,e_\theta,e_z$ are all $2\pi$ periodic in $\theta$, trivially for the constant vector $e_z$, and therefore we have the Fourier series
    \begin{align}
    u_r(r,z,\theta)&= \label{A}
    u_{r,0}(r,z)
    +\sum_{k=1}^\infty u_{r,k}(r,z)\cos(k\theta) \\
    u_z(r,z,\theta)&= \label{B}
    u_{z,0}(r,z)
    +\sum_{k=1}^\infty u_{z,k}(r,z)\cos(k\theta) \\
    u_\theta(r,z,\theta)&= \label{C}
    \sum_{k=1}^\infty u_{\theta,k}(r,z)\sin(k\theta) 
    \end{align}

    We know that $u=u^{M_{e_3}}$, which implies that $u_r,u_\theta$ are even with respect to $z$ and $u_z$ is odd with respect to $z$. 
    This in turn implies that
    for all $k\in\mathbb{Z}^+,
    u_{r,k},u_{\theta,k}$ are even in $z$
    and $u_{z,k}$ is odd in $z$.
    It remains to show that the only Fourier modes in the expansion are of order $k=6n$. Observe the unit basis vectors in cylindrical coordinates are invariant under rotation with
    \begin{align}
    e_r^{R_\frac{\pi}{3}}&=e_r \\
    e_\theta^{R_\frac{\pi}{3}}&=e_\theta,
    \end{align}
    and furthermore that
    \begin{equation}
    R_\frac{\pi}{3}^{tr}(r,z,\theta)=
    \left(r,z,\theta-\frac{\pi}{3}\right).
    \end{equation}
    Using the fact that $u=u^{R_\frac{\pi}{3}}$,
    we can conclude that
    \begin{multline}
    u(x)= u_{r,0}(r,z)e_r+u_{z,0}(r,z)e_z \\
    +\sum_{k=1}^\infty
    u_{r,k}(r,z)
    \cos\left(k\theta-\frac{k\pi}{3}\right) e_r
    +u_{z,k}(r,z)
    \cos\left(k\theta-\frac{k\pi}{3}\right) e_z
    +u_{\theta,k}(r,z)
    \sin\left(k\theta-\frac{k\pi}{3}\right) e_\theta.
    \end{multline}
    This implies that
    \begin{align}
    u_r(r,z,\theta)&= \label{X}
    u_{r,0}(r,z)
    +\sum_{k=1}^\infty
    u_{r,k}(r,z)\left(\cos(k\theta)
    \cos\left(\frac{k\pi}{3}\right)
    +\sin(k\theta)
    \sin\left(\frac{k\pi}{3}\right)\right) \\
    u_z(r,z,\theta)&= \label{Y}
    u_{z,0}(r,z)
    +\sum_{k=1}^\infty
    u_{z,k}(r,z)\left(\cos(k\theta)
    \cos\left(\frac{k\pi}{3}\right)
    +\sin(k\theta)
    \sin\left(\frac{k\pi}{3}\right)\right) \\
    u_\theta(r,z,\theta)&= \label{Z}
    \sum_{k=0}^\infty
    u_{\theta,k}(r,z)\left(
    \sin(k\theta)\cos\left(\frac{k\pi}{3}\right)
    -\cos(k\theta)\sin\left(\frac{k\pi}{3}\right)
    \right).
    \end{align}
    Using the uniqueness of Fourier series and comparing \cref{A,B,C} and \cref{X,Y,Z}, we can see that 
    for all $k\in\mathbb{N}$,
    \begin{equation}
    u_{r,k},u_{z,k},u_{\theta,k}=0,
    \end{equation}
    unless
    \begin{align}
    \cos\left(\frac{k\pi}{3}\right)&=1 \\
    \sin\left(\frac{k\pi}{3}\right)&=0.
    \end{align}
    This happens if and only if $k\in 6\mathbb{N}$, so letting $k=6n$, we can see that
    \begin{multline}
    u(x)=
    u_{r,0}(r,z)e_r
    +u_{z,0}(r,z)e_z \\
    +\sum_{n=1}^{\infty}
    u_{r,n}(r,z)\cos(6n\theta)e_r
    +u_{z,n}(r,z)\cos(6n\theta)e_z
    +u_{\theta,n}(r,z)
    \sin(6n\theta)e_\theta.
    \end{multline}

    For the forward direction of the proof, all that remains is to state the divergence free constraint in terms of the Fourier series. In cylindrical coordinates we have
    \begin{equation}
    \nabla=
    e_r\partial_r+e_z\partial_z
    +\frac{1}{r}e_\theta \partial_\theta.
    \end{equation}
    Therefore we can compute that
    \begin{equation}
    \nabla\cdot u(x)
    =
    \partial_r u_{r,0}
    +\frac{1}{r}u_{r,0}
    +\partial_z u_{z,0}
    +\sum_{n=1}^\infty
    \left(\partial_r u_{r,n}
    +\frac{1}{r}u_{r,n}
    +\partial_z u_{z,n}
    +\frac{6n}{r}u_{\theta,n}(r,z)\right)
    \cos(6n\theta),
    \end{equation}
    and so by the uniqueness of Fourier series
    \begin{equation}
    \partial_r u_{r,0}
    +\frac{1}{r}u_{r,0}
    +\partial_z u_{z,0}=0,
    \end{equation}
    and for all $n\in\mathbb{N}$,
    \begin{equation}
    \partial_r u_{r,n}
    +\frac{1}{r}u_{r,n}
    +\partial_z u_{z,n}
    +\frac{6n}{r}u_{\theta,n}(r,z)=0. 
    \end{equation}
    This completes the forward direction of the proof.

    Now we consider the reverse direction. Suppose $u$ can be expressed as a Fourier series in $\theta$ satisfying \cref{Fourier,DivFreeBase,DivFreeMode}.
    Then by the exact computations above we can see that 
    \begin{equation}
    \nabla\cdot u=0,
    \end{equation}
    and furthermore that $M_{e_3},M_{e_2},
    R_\frac{\pi}{3}$ are symmetries of $u$.
    Observing that
    \begin{equation}
    M_{e_1}= M_{e_2}R_\frac{\pi}{3}^3,
    \end{equation}
    we can see that $M_{e_1}$ is a symmetry of $u$, and therefore, applying \Cref{GenerateThm}, $u$ is $\mathcal{G}$-symmetric.
\end{proof}

\begin{corollary} \label{AxisymmInclusion}
    Suppose $u\in H^s_{df}, s>\frac{5}{2}$ is axisymmetric, swirl-free, mirror symmetric in the $e_3$ direction---meaning that $u_r$ is even in $z$ and $u_z$ is odd in $z$.
    Then $u$ is $\mathcal{G}$-symmetric.
\end{corollary}

\begin{proof}
    This follows immediately from \Cref{FourierSeries}, as a vector field of this form is just the zero order term in the Fourier series expansion for any $\mathcal{G}$-symmetric vector field.
\end{proof}

\begin{corollary} \label{FourierSeriesCor}
    Suppose $u\in C\left([0,T_{max});H^s_{df}\right), s>\frac{5}{2}$ is a solution of the Euler with $\mathcal{G}$-symmetric initial data $u^0\in H^s_{df}$. Then  $u(\cdot,t)$ is 
    $\mathcal{G}$-symmetric for all $0\leq t<T_{max}$; consequently, for all $0\leq t<T_{max}$,
    \begin{multline}
    u(x,t)=
    u_{r,0}(r,z,t)e_r
    +u_{z,0}(r,z,t)e_z \\
    +\sum_{n=1}^{\infty}
    u_{r,n}(r,z,t)\cos(6n\theta)e_r
    +u_{z,n}(r,z,t)\cos(6n\theta)e_z
    +u_{\theta,n}(r,z,t)
    \sin(6n\theta)e_\theta,
    \end{multline}
    where for all $n\in\mathbb{Z}^+,
    u_{r,n},u_{\theta,n}$ are even in $z$
    and $u_{z,n}$ is odd in $z$,
    and furthermore
    \begin{equation}
    \partial_r u_{r,0}
    +\frac{1}{r}u_{r,0}
    +\partial_z u_{z,0}
    =0,
    \end{equation}
    and for all $n\in\mathbb{N}$,
    \begin{equation}
    \partial_r u_{r,n}
    +\frac{1}{r}u_{r,n}
    +\partial_z u_{z,n}
    +\frac{6n}{r}u_{\theta,n}
    =0.
    \end{equation}
\end{corollary}

\begin{proof}
    Fix $Q\in \mathcal{G}$, and let $v=u^Q$.
    It is classical that $v\in C\left([0,T_{max});H^s_{df}\right)$ is also a solution of the Euler equation, because the solution set of the Euler equation is invariant under orthogonal transformations. See section 1.1 in \cite{MajdaBertozzi} for details.
    By hypothesis we can see that $u^0=v^0$, and so by uniqueness we can determine that $u=v$.
    Therefore, for all $Q\in \mathcal{G}$,
    \begin{equation}
    u^Q=u,
    \end{equation}
    and so $u$ is $\mathcal{G}$-symmetric.
    Applying \Cref{FourierSeries}, this completes the proof.
\end{proof}

\begin{remark}
    We can also use the Fourier series expansions in \Cref{FourierSeries} and \Cref{FourierSeriesCor} for $\mathcal{G}_\sigma$-symmetric solutions of the incompressible Euler equation if we take
    \begin{align}
    z &= x\cdot \Tilde{\sigma} \\
    x' &= x-z\Tilde{\sigma} \\
    r &=  |x'| \\
    e_z &= \Tilde{\sigma} \\
    e_r &= \frac{x'}{|x'|} \\
    e_\theta &= e_z \times e_r,
    \end{align}
    and we take $-\pi<\theta\leq \pi$ to be the angle in the plane $x_1+x_2+x_3=0$ plane between $x'$ and $\left(\begin{array}{c}
    \frac{1}{\sqrt{2}} \\ 
    -\frac{1}{\sqrt{2}} \\ 0
    \end{array}\right)$.
\end{remark}

\begin{remark}
    Note that \Cref{FourierSeriesCor} clearly implies that the axisymmetric, swirl free solution of the Euler equation $u\in C\left([0,1);C^{1,\alpha}\right)
    \cap C^1\left([0,1);C^\alpha\right)$ constructed by Elgindi \cite{Elgindi} blowing up in finite-time at $T_{max}=1$ is $\mathcal{G}$-symmetric.
    It then immediately follows from \Cref{EquivThm}
    that $u^{Q_\sigma} \in 
    C\left([0,1);C^{1,\alpha}\right)
    \cap C^1\left([0,1);C^\alpha\right)$ 
    is $\mathcal{G}_\sigma$-symmetric and blows up in finite-time $T_{max}=1$, so finite-time blowup in this geometric setting is clearly possible for $C^{1,\alpha}$ solutions of the Euler equation. This gives a quicker proof of \Cref{AlphaPermuteBlowup}, but the work in \Cref{PermuteAxisymSubsection} was necessary to develop the relationship between permutation symmetry and axisymmetry for vector fields that are not necessarily odd or $\sigma$-mirror symmetric.
\end{remark}

\begin{remark}
    The Ansatz from \Cref{FourierSeriesCor} is likely to be useful in further study of the finite-time blowup problem for the Euler equation. In particular, because the $C^{1,\alpha}$ blowup solution from \cites{Elgindi,ElgindiGhoulMasmoudi} is axisymmetric, swirl-free, and therefore has a geometric structure that rules out finite-time blowup for smooth solutions, it would be progress to find even other $C^{1,\alpha}$ solutions of the Euler equation in the $\mathcal{G}$-symmetric geometry without axisymmetry. In this case, it is possible that the blowup could survive mollification of the initial data, which it cannot possibly do for axisymmetric, swirl-free, $C^{1,\alpha}$ blowup solutions. A natural question to consider is whether a small, non-axisymmetric perturbation of the of the blowup solution constructed by Elgindi will still exhibit finite-time blowup.
\end{remark}

\section*{Acknowledgements}

This material is based upon work supported by the Swedish Research Council under grant no. 2021-06594 while the author was in residence at the Institut Mittag-Leffler in Djursholm, Sweden during the Autumn 2023 semester.

\bibliography{bib}

\end{document}